\documentclass{amsart}

\usepackage{amssymb, amsthm, amsmath,  amssymb, amsbsy,   amsfonts,
  latexsym,  amsopn,   amstext, amsxtra,  euscript,   amscd} %, amsaddr}
\usepackage[shortlabels]{enumitem}
\usepackage{graphicx}
\usepackage{cite}
\usepackage{hyperref}
\usepackage{longtable}
\usepackage[usenames,dvipsnames]{color}

\hypersetup{
  colorlinks   = true, %Colours links instead of ugly boxes
  urlcolor     = blue, %Colour for external hyperlinks
  linkcolor    = blue, %Colour of internal links
  citecolor   = red %Colour of citations
}

\usepackage{array}
\usepackage{multirow}

\usepackage[capitalise]{cleveref}
\usepackage{aliascnt}

\newtheorem{thm}{Theorem}

\newtheorem{lem}[thm]{Lemma}
\newtheorem{cor}[thm]{Corollary}
\newtheorem{prop}[thm]{Proposition}

\theoremstyle{definition}

\newtheorem{question}[thm]{Question}
\newtheorem{rem}[thm]{Remark}

\crefname{thm}{Thm.}{}
\crefname{prop}{Prop.}{}
\crefname{lem}{Lem.}{}
\crefname{cor}{Cor.}{}
\crefname{prob}{Problem}{}
\crefname{figure}{Fig.}{}

\numberwithin{thm}{section}
\usepackage[citation-order, msc-links, backrefs]{amsrefs}

\DeclareMathOperator\Hom{Hom }
\DeclareMathOperator\End{End }

\DeclareMathOperator\Spec{Spec }
\DeclareMathOperator\Sym{Sym}
\DeclareMathOperator\ord{ord }

\def\iso{\cong}

\def\lar{\longrightarrow}

\def\X{\mathcal C}

\DeclareMathOperator\Aut{Aut}

\DeclareMathOperator\Jac{Jac }

\def\<{\langle} \def\>{\rangle}
\def\mc#1{\mathcal{#1}}
\def\mf#1{\mathfrak{#1}}
\def\ol#1{\overline{#1}}

\def\P{\mathbb P}
\def\C{\mathbb C}
\def\Q{\mathbb Q}

\def\Z{\mathbb Z}

\def\nats{\mathbb N}
\def\ints{\mathbb Z}
\def\rats{\mathbb Q}

\def\complex{\mathbb C}
\def\proj{\mathbb P}

\def\X{\mathcal C}

\def\H{\mathcal H}

\def\p{\mathfrak p}

\def\M{\mathcal M}

\def\Y{\mathcal Y}

\def\AA{\mathcal A}
\def\B{\mathcal B}
\def\Hom{\mbox{Hom} }

\title{Superelliptic curves with many automorphisms and  CM Jacobians}
\author{Andrew Obus}
\address{1 Bernard Baruch Way \\
  New York, NY 10010}
\thanks{The first author was supported by the National
  Science Foundation under DMS Grant No.\ 1900396}

\author{Tanush Shaska}
\address{146 Library Drive\\
368 Mathematics Science Center\\
Rochester MI 48309-4479
}

\begin{document}
\begin{abstract}
Let $\X$ be a smooth, projective, genus $g\geq 2$ curve, defined over
$\C$.  Then $\X$ has \emph{many   automorphisms} if its corresponding
moduli point $\p \in \M_g$ has a neighborhood $U$ in the complex
topology,  such that all curves corresponding to points in $U
\setminus \{\p \}$ have  strictly fewer automorphisms than $\X$.   We
compute completely the list of superelliptic curves having many automorphisms.  For  each of these curves, we determine whether its Jacobian has complex multiplication.  
As a consequence, we prove the converse of Streit's complex multiplication criterion 
for these curves.
\end{abstract}



\maketitle

%********************************************
\section{Introduction}
An abelian variety $\AA$ has complex multiplication (or is of CM-type) over a field $k$ if $\End_k^0 (\AA)$ contains a commutative, semisimple $\Q$-algebra of dimension $2\dim \AA$.   They were first studied by M. Deuring \cite{deuring, deuring-2} for elliptic curves and generalized to Abelian varieties by Shimura and Tanyama in \cite{shimura-tanyama}. By abuse of terminology, a curve is said to have complex  multiplication (or to be of CM-type)  when its Jacobian is of CM-type.   Since the CM property is a property of the Jacobian, it is an invariant of the curve. A natural question is whether there is anything special about the points in the moduli space $\M_g$ of genus $g\geq 2$ curves, for which the Jacobian is of CM-type; see \cite{Oort}.   F. Oort asked if curves with many automorphisms (cf. \cref{many-aut}) are all of CM-type.  The answer to this question is negative, as explained in  \cite{oort-2},  where a full history of the problem and recent developments are given. 

Let $\X$ be a smooth, projective, genus $g\geq 2$ curve defined over $k$,   $\p \in \M_g$  its corresponding moduli point, and  $G:=\Aut (\X)$  the automorphism group of $\X$ over the algebraic closure of $k$.  For our purposes we will assume $k=\C$.  We say that $\X$ has \emph{many   automorphisms} if its corresponding point $\p \in \M_g$ has a neighborhood $U$ (in the complex topology) such that all curves corresponding to points in $U \setminus \{\p \}$ have automorphism group strictly smaller than $G$.   They shouldn't be confused with curves with \emph{large automorphism group} which are curves with automorphism group $|G| > 4 (g-1)$.  Not all curves with large automorphism group are curves with many automorphisms.

As mentioned above, Oort asked if such curves are of CM-type, and this
is not true in general.  However it remains an interesting question to determine which curves with many automorphisms are of CM-type. 
 In general, for a given $g\geq 2$ we can determine the full list of automorphism groups that occur; see \cite{kyoto} and \cite{aut} for a complete survey on automorphism groups of algebraic curves.  It is difficult from the group   $\Aut (\X)$  alone to determine if $\X$ is of CM-type without knowing anything about the equation of the curve.  However, there is only one class of curves for which we can determine the equation of the curves explicitly starting from the automorphism group, namely the  superelliptic curves.  Hence, it is a natural choice to try to determine which superelliptic curves with many automorphisms are of CM-type. 

In \cite{muller-pink}, the authors solved this problem for
hyperelliptic curves.   Their main tool is a formula of Streit from
\cite{streit-01}, which gives   conditions on the characters of the
group of automorphisms of the curve.  More precisely,  let $\chi_\X$
be the character of $\sigma$  on $H^0(\X, \omega_{\X})$, and let
$\Sym^2\chi_\X$ be the character of $\sigma$ on $\Sym^2H^0(\X,
\omega_{\X})$.  By $\chi_{triv}$ we denote the character of the
trivial representation on $\C$.  Streit showed that if $\< \Sym^2
\chi_\X, \chi_{triv} \> =0$ then $\Jac \X$ has complex multiplication;
see \cite{streit-01}.  We say that $\X$ \emph{satisfies Streit's
  criterion} if this inner product is $0$.

For $\X$  hyperelliptic, the authors in \cite{muller-pink}  determine
a formula which computes $\Sym^2 \chi_\X$ and through this formula are
able to determine precisely if a
hyperelliptic curve with many automorphisms is or is not of CM-type.
They prove their formula using the fact that it is easy to write a monomial basis of
holomorphic differentials for hyperelliptic function fields. 
As a consequence, the converse of Streit's criterion holds for hyperelliptic curves with many automorphisms and reduced automorphism group isomorphic to  $A_4$, $S_4$, or $A_5$.    In other words, no such curve that fails Streit's criterion can have complex multiplication.

Using a similar approach we are able to prove a similar formula for
superelliptic curves (cf. \cref{Pstreit}).    Let $\X:  y^n =
f(x)$ be a smooth superelliptic curve  defined over $\complex$ (in
particular, $f(x)$ is assumed to be a \emph{separable polynomial}, see
 \cref{Smany}), and let $\tau$ be the order $n$ automorphism given by
$y \mapsto \zeta_ny$, where $\zeta_n$ is a primitive $n$th rooth of
unity.   Let $G$ be the normalizer of $\tau$ in $\Aut(\X)$, and let
$\ol{G} = G/\langle \tau \rangle$, which naturally lies in
$\Aut(\proj^1) \cong PGL_2(\complex)$.  For each $\ol{\sigma} \in \ol{G}$, let $m$ be
its order, and let $\zeta_{\ol{\sigma}}$ be either ratio of the
eigenvalues when  $\ol{\sigma}$ is thought of as an element of
$PGL_2(\complex)$.  Observe that $\zeta_{\ol{\sigma}}$ is a primitive
$m$th root of unity.
Let $\zeta_{n,
  \ol{\sigma}}$ be a primitive $mn$th root of unity such that
$\zeta_{n, \ol{\sigma}}^n = \zeta_{\ol{\sigma}}$.  Define
\[
k_{\ol{\sigma}} =
\left\{
\begin{split}
& 1     \text{  if   }    \ol{\sigma}  \; \text{ fixes a branch point
  of   }  \, \X \to  \proj^1 \\
 &  0 \quad \text{ otherwise.}  \\
\end{split} 
\right.
\]
Let
$A \subseteq \ints_{\geq 0} \times \ints_{\geq 0}$ be the set of ordered pairs defined below in (\cref{Eholforms}).  If
\begin{small}
\[
\sum_{\ol{\sigma} \in \ol{\Aut} (\X)} \sum_{i=0}^{n-1}    \left(  \sum_{(a,b) \in A} \zeta_{\ol{\sigma}}^{2(a+1)} \zeta_n^{2(b+1)i}  \zeta_{n, \ol{\sigma}}^{2(b-n+1)k_{\ol{\sigma}}} +  \left(\sum_{(a,b) \in A}   \zeta_{\ol{\sigma}}^{a+1}  \zeta_n^{(b+1)i}     \zeta_{n, \ol{\sigma}}^{(b-n+1)k_{\ol{\sigma}}}  \right)^2 \right)
\]
\end{small}
vanishes, then $\X$ satisfies Streit's criterion, and thus has CM.   However, not all
superelliptic curves with many automorphisms satisfy Streit's criterion; see
\cref{Tcurvelist2}.  We in fact prove the converse of Streit's
criterion for superelliptic curves with many automorphisms; any such
curve not satisfying Streit's criterion does not have CM \cref{Cmain}.  To do this we use stable reduction (cf. \cref{Sstable}) and the theory of semistable models as in
\cite{BW}, as well as the so-called \emph{criterion of M\"{u}ller-Pink}; see \cref{Sfrobenius}. Our
computations were done using Sage and GAP and are made available along with the arxiv posting of this paper.  

It is still an open question whether these results can be generalized  to larger families of curves, for example for  generalized   superelliptic curves as  listed in \cite{HQS}.  

In general, it is believed that curves with extra automorphisms give
Jacobian varieties with a large number of endomorphisms (here
``large'' is used informally).  Hence, perhaps an interesting family of curves to investigate would be 
all curves with large automorphism groups and many automorphisms;
these were determined in \cite{kyoto}.
It would also be interesting to obtain a theoretical explanation for \cref{Cmain}.

\section{Preliminaries} 
Throughout, we work over the field $\complex$.  An Abelian variety is an absolutely irreducible projective variety which is  a group scheme. A morphism of  Abelian varieties  $\AA$ to   $\B$ is a \emph{homomorphism} if and only if it maps the identity element of $\AA$ to the identity element of $\B$.

Let $\AA$, $\B$ be abelian varieties.  We denote the $\Z$-module of homomorphisms  $\AA \mapsto  \B$  by $\Hom( \AA, \B)$  and the ring of endomorphisms $\AA \mapsto \AA$ by $\End \AA$.    It is  more convenient  to   work with the $\Q$-vector spaces $\Hom^0 (\AA, \B):= \Hom(\AA, \B) \otimes_\Z \Q$, and $ \End^0 \AA:= \End \AA\otimes_\Z \Q$.   Determining $\End \AA$ or $\End^0 \AA$ is an interesting problem on its own; see  \cite{Oort}.

The ring of endomorphisms of a generic Abelian variety $\AA$ over $\complex$
is ``as small
as  possible", that is, $\End(\AA)=\Z$ in
general. 
In general, $\End^0(\AA)$
is a  $\Q$-algebra of dimension at most $4\dim (\AA)^2$. Indeed,
$\End^0 (\AA)$ is a semi-simple algebra, and by duality  one can apply
a complete classification due to Albert of \emph{possible} algebra
structures on $\End^0(\AA)$, see \cite[pg. 202]{Mum}.  We say that an abelian variety $\AA$ has \emph{complex multiplication} if $\End^0 (\AA)$ contains a commutative, semisimple $\Q$-algebra of dimension $2\dim \AA$. 

For the following elementary result, see \cite{shimura-tanyama}. 

\begin{lem}\label{LCMsub}
If $\AA$ is an Abelian variety with CM, then every abelian subvariety of $\AA$ also has CM.
\end{lem}

 \subsection{Curves and their Jacobians}  
Throughout, $\X$ is a smooth, projective curve defined over $\complex$. We will denote its group of automorphisms by $\Aut(\X)$.   It is a well known fact that % when $\char k =0$, then
$|\Aut (\X) | \leq 84 (g-1)$.    One gets a stratification of $\M_g$  by strata of curves with the same automorphism group, and the generic  curve of genus $g  >2$ has trivial automorphism group.

We state the following Corollary of \cref{LCMsub}, which is used  repeatedly.

\begin{cor}\label{CCMquotient}
If $\psi: \X \to \Y$ is a finite morphism of curves and $\Jac \X$ has   CM, then so does $\Jac \Y$.
\end{cor}

\begin{proof}
Since $\psi^*$ gives an embedding of $\Jac \Y$ in $\Jac \X$, the   corollary follows from \cref{LCMsub}.
\end{proof}

%*****************************************
\subsection{Hurwitz spaces}
We consider finite covers $\eta: \X \to \P^1$ of degree $n$. Then $\eta^*$ identifies  $\complex(\P^1)=:\complex(x)$ with a subfield of $\complex (\X)$.  First, we introduce the equivalence: $\eta \sim \eta'$ if there are isomorphisms $\alpha:\X\rightarrow \X'$ and $\beta\in \Aut(\P^1)$ with
\[\beta\circ \eta=\eta'\circ \alpha.\]
The \emph{monodromy group} of $\eta$ is the Galois group of the Galois closure $L$ of $\complex(\X)/\complex(x)$. We embed $G$ into $S_n$, the symmetric group with $n$ letters,  and fix the ramification type of the covers $\eta$. We assume that exactly $r\geq 3$ points in $\P^1(k)$ are  ramified (i.e. their preimages contain fewer than $n$ points).  Note that the ramification groups are cyclic. 

By the classical theory of covers of Riemann surfaces, which can be  transferred to the algebraic setting by the results of Grothendieck, it follows that there is a tuple $(\sigma_1, \dots , \sigma_r)$ in $S_n$ such that $\sigma_1 \cdot \cdot \cdot \sigma_r =1$, $\ord(\sigma_i)=e_i$ is the ramification order of the $i$-th ramification point $P_i$ in $L$  and $G:=\< \sigma_1, \dots , \sigma_r\> $ is a transitive group in $S_n$. 
We call such a tuple the \emph{signature} $\sigma$ of the covering $\eta$ and remark that such tuples are determined up to conjugation in $S_n$, and that the genus of $\X$ is determined by the signature because of the Hurwitz genus formula.

Let $\H_\sigma$ be the set of pairs $([\eta], (p_1, \dots , p_r))$, where $[\eta]$ is an equivalence class of covers of type $\sigma$, and $p_1, \dots , p_r$ is an ordering of the branch points of $\phi$ modulo automorphisms of $\P^1$.
The set $\H_\sigma$ carries the structure of a scheme; in fact it is a  quasi-projective variety called the \emph{Hurwitz space}.   We have the forgetful morphism 
\[\Phi_\sigma: \H_\sigma \to \M_g\]
 mapping $([\eta], (p_1, \dots , p_r))$ to the isomorphism class $[\X]$ in the moduli space $\M_g$.   Each component of $\H_\sigma$ has the same image in $\M_g$.

Define the \emph{moduli dimension of $\sigma$} (denoted by $\dim (\sigma)$) as the dimension of $\Phi_\sigma(\H_\sigma)$; i.e., the dimension of the locus of genus $g$ curves admitting a cover to $\P^1$ of type $\sigma$. We say $\sigma$ has   \emph{full moduli dimension} if   $\dim(\sigma)=\dim \M_g$; see \cite{kyoto}.
 
%************
\subsection{Curves with many automorphisms}\label{many-aut}
Let $\X$ be a genus $g\geq 2$ curve defined over $\C$,  $\p \in \M_g$ its corresponding moduli point, and $G:=\Aut (\X)$.  
We say that $\X$ has \emph{many automorphisms} if $\p \in \M_g$ has a neighborhood $U$ (in the complex topology) such that all curves corresponding to points in $U \setminus \{\p \}$ have automorphism group strictly smaller than $\p$.

\begin{lem}[{\cite[Lemma 4.4]{oort-2}} or {\cite[Theorem 2.1]{muller-pink}}]\label{lem-2-4}
Let $\X$ have genus $\geq 2$ as above, let $G:=\Aut (\X)$, and let $\eta : \X \rightarrow \P^1$ the corresponding map with signature $\sigma$. Then the   following are equivalent:
\begin{enumerate}
\item  $\X$ has many automorphisms. 
\item  There exists a subgroup $H < G$ such that $g \left(  \X/H \right) = 0$ and $\X \to \X/H$ has exactly three branch points.  
\item  The quotient $\X/G$ has genus 0 and $\X \to \X/G$ has exactly
  three branch points.  
\item The signature $\sigma$ has moduli dimension 0. 
\end{enumerate}

\end{lem}

\begin{question}[F. Oort]
If $\X$ has many automorphisms, does $\Jac \X$ have complex multiplication? 
\end{question}

Wolfart answered this question  for all curves of genus $g \leq 4$;  see \cite[\S5]{wolfart}.

%******************************************************************************************
\section{Superelliptic curves with many automorphisms}\label{Smany}
The term \emph{superelliptic curve} has been used differently by many
authors.
%Some use it to mean any smooth projective curve $\X$
%admitting an $H$-cover $\pi: \X \to \proj^1$, with $H
%\cong \ints/n$, but it is perhaps more common to require the affine
%equation of $\X$ to be
Most use it to mean a smooth projective curve $\X$ with affine
equation of the form
$y^n = f(x)$, where $f(x) \in \complex[x]$ has discriminant $\Delta
(f)\neq 0$.  If $H := \langle \tau \rangle
\subseteq \Aut(\X)$ is the subgroup generated by $\tau(y) = \zeta_ny$, then it is sometimes further required that $H$ be normal (or central)
in $\Aut(\X)$.  

We will follow the definition in \cite{HQS}.  Specifically, a
\emph{superelliptic curve} is a smooth projective curve $\X$ of genus
$\geq 2$ with
affine equation $y^n = \prod_{i=1}^r (x - a_i)$, with the $a_i$ distinct
    complex numbers such that
\begin{enumerate}[\upshape (i)]
  \item If $H$ is as above,
then $H$ is \emph{normal} in $\Aut(\X)$.
   \item Either $n \mid r$ or $\gcd(n, r) = 1$ (this guarantees that
     all branch points have index $n$).
\end{enumerate}

\begin{rem}
In fact, if $\X$ is a superelliptic curve with many automorphisms, we
have that $n \mid r$ or $ r \equiv -1 \pmod{n}$, see  
\end{rem}

If $\X$, $H$, and $\tau$ are
as above, we call $\tau$ an
\emph{superelliptic automorphism (of level $n$)} and $H$ a
\emph{superelliptic group (of level $n$)} of $\X$.  

Suppose $\X$ is a superelliptic curve,
with superelliptic group $H$ and corresponding $H$-cover $\pi: \X
\to \proj^1$.  If $G = \Aut(\X)$, then there is a short
exact sequence $1 \to H \to G \to \overline{G} \to 1$, where
$\overline{G}$ is a group of M\"obius transformations keeping the
set of branch points of $\pi$ invariant.  We call $\ol{G}$ the
\emph{reduced automorphism group} of $\X$.

We also define a \emph{pre-superelliptic curve} to be a curve
satisfying all the requirements of a superelliptic curve except
possibly for
(i) above.  In this case, if $N$ is the normalizer of $H$ in
$\Aut(\X)$, then we have a similar exact sequence $1 \to H \to N \to
\overline{N} \to 1$, and we call $\ol{N}$ the reduced
automorphism group of $\X$.  In this case, $H$ is called a
\emph{pre-superelliptic group} and $\tau$ is called a
\emph{pre-superelliptic automorphism}.

Because verifying that a curve is
pre-superelliptic does not depend on computing its entire automorphism
group, it can be significantly easier than verifying that a curve is superelliptic.

As an immediate
consequence of \cite[Lem.\ 1]{HQS}, we obtain the following proposition.

\begin{prop}\label{Pnormalcentral}
If $\X$ is a pre-superelliptic curve, then any pre-superelliptic group is in
    fact \emph{central} in its normalizer.
  \end{prop}
  
   Superelliptic curves over \emph{finite} fields were described up to isomorphism in \cite{Sa}.
  For more on arithmetic aspects of such curves we refer to \cite{m-sh}. 

%*********************************************************
\subsection{Superelliptic curves with many automorphisms}
In the rest of this subsection, we construct a list containing all superelliptic curves with many automorphisms.  
As above, let $\X$ be a pre-superelliptic
curve defined over $\C$ with automorphism group $G := \Aut(\X)$, and
pre-superelliptic group $H$ generated by $\tau$ of order $\geq 2$.
Let $N \subseteq G$ be the normalizer of $H$,
and let $\ol{N} := N/H$ be the reduced automorphism
group. 
In fact, what we construct is the list of all pre-superelliptic curves
$\X$ as above such that the quotient morphism $\X \to \X/N$ is branched at exactly three
points.  Since $N = G$ for a superelliptic curve, this list contains
all superelliptic curves $\X$ such that $\X \to \X/G$ is branched at
exactly three points, thus, by \cref{lem-2-4}, all
superelliptic curves with many automorphisms.

\begin{prop} \label{prop-1}
Let $\X$ be a pre-superelliptic curve, let $N \subseteq \Aut(\X)$ be the normalizer of the
a pre-superelliptic group $H$ of level
$n$, and let $\ol{N} = N/ H$.  Then $\ol{N}$ is isomorphic to
either $C_m$, $D_{2m}$, $A_4$, $S_4$, or $A_5$.  If the quotient map $\Phi: \X \to \X / N$
is branched at exactly three points, then furthermore:
\begin{enumerate}[\upshape (i)]
\item If $\ol{N} \iso C_m$  then $\X$ has equation $y^n=x^m+1$ or $y^n=x(x^m+1)$. 
\item If $\ol{N} \iso D_{2m}$  then $\X$ has equation  $y^n=  x^{2m}-1$ or $y^n= x (x^{2m}-1)$.
\item If $\ol{N}\iso A_4 $  then $\X$ has equation  $y^n = f(x)$    where $f(x)$ is  the following 
\[y^n = x (x^4-1)  (x^4 + 2i \sqrt{3} x^2 +1). \]
Furthermore, the $A_4$-orbit of $\infty$ consists of itself and the roots of $x(x^4-1)$.
\item If $\ol{N} \iso S_4 $   then $\X$ has equation $y^n = f(x)$   where $f(x)$ is one of the following: 
$r_4(x)$,  $s_4(x)$,  $t_4(x)$, $r_4(x)s_4(x)$,  $r_4(x)t_4(x)$,  $s_4(x)t_4(x)$,  $r_4(x)s_4(x)t_4(x)$,    
where $r_4$, $s_4$, $t_4$ are as in \cref{S4-pols}.   Furthermore, the $S_4$-orbit of $\infty$ consists of itself and the roots of $t_4(x)$.
\item If $\ol{N} \iso A_5$ then $\X$ has equation $y^n = f(x)$ where $f(x)$ is one of the following:   
$r_5(x)$, $ s_5(x)$,  $ t_5(x)$, $ r_5(x)s_5(x)$,    $ r_5(x)t_5(x)$, $ s_5(x)t_5(x)$, $r_5(x)s_5(x)t_5(x)$, 
where $r_5$, $s_5$, $t_5$ are as in \cref{A5-pols}.  Furthermore, the $A_5$-orbit of $\infty$ consists of itself and the roots of $s_5(x)$.
\end{enumerate}
\end{prop}

\begin{rem}
The notation $r_4$, $s_4$, $t_4$, $r_5$, $s_5$, $t_5$ in \cref{prop-1}  above is consistent with that used in \cite{muller-pink}.
\end{rem}

\begin{proof}
The first statement holds because $\ol{N} \subseteq PGL_2(\C)$, using
the well-known classification of finite subgroups of $PGL_2(\C)$.

We have the following diagram:
\[\Phi: \X  \buildrel{H} \over \lar  \P^1_x \buildrel {\ol{N}} \over
  \lar \P^1_z = \X/N.\]
The group $\ol{N}$ is the monodromy group of the cover $\phi:\P^1_x
\to \P^1_z$.
Let $y^n = f(x)$ be the
equation of $\X$, where $f(x)$ is a separable polynomial.  Now, the
map $\phi$ is given by the rational function $z$ in $x$, which 
has degree $|\ol{N}|$.

Let $S = \{q_1, q_2, q_3\}$ be the set of branch points of $\Phi: \X
\to \P^1_z$, let $W$ be the set of branch points of $\pi: \X \to \P^1_x$ (that is,
the roots of $f$ and possibly $\infty$).  Since $\X \to \P^1_z$ is
Galois and $|H| \geq 2$, there
exists a non-empty set $T \subseteq S$ such that $W =
\phi^{-1}(T)$.

Write the rational function $z$ in lowest terms as a ratio of
polynomials $\frac{\Psi(x)}{\Upsilon(x)}$.
We write $z-q_i=\frac{\Gamma_i (x)}{\Upsilon(x)}$ in lowest terms, for each branch point $q_i$, $i=1,2,3$, where $\Gamma_i (x) \in \complex[x]$. Hence,
\[\Gamma_i   (x)=\Psi(x)-q_i\cdot\Upsilon(x)\] 
is a degree $|\ol{N}|$ polynomial and the multiplicity of each root of
$\Gamma_i (x)$ corresponds to the ramification index for each $q_i$
(if $q_i = \infty$, then $\Gamma_i(x) := \Upsilon(x)$).  The roots of
$\Gamma_i(x)$ are the preimages of $q_i$ under $\phi$.  So letting
$\gamma_i(x) = \text{rad}(\Gamma_i(x))$, we conclude that the equation of
$\X$ is given by $y^n = f(x)$, where
\begin{equation}\label{Eroots}
  f(x) = \prod_{q_i \in T} \gamma_i(x).
\end{equation}

The rest of the proof proceeds similarly to \cite[\S4]{san-2}.  In
particular, for $$\ol{N} \in \{C_m, D_{2m}, A_4, S_4, A_5\},$$ we can make a change of
variables in $x$ and $z$ so that $z = \phi(x)$ is given by the
appropriate entry in the first 5 rows of \cite[Table 1]{san-2}.  We
now go case by case.

 \textbf{i)   $\ol{N} \cong C_m$:}   In this case, $$\phi(x) = x^m,$$
 which has branch points $q_1 = \infty$ and $q_2 = 0$.
 Hence $\gamma_1(x)=1$ and $\gamma_2(x)=x$.  After a change of
 variables in $z$, we may assume without loss of generality that $q_3
 = -1$, which yields $\gamma_3(x) = x^m + 1$.  Since
the covering $\X \lar \P^1_z$ has three branch points, we must have
$q_3 \in T$.   From \cref{Eroots}, we have $f(x)=x^m+1$ or $f(x) =
x(x^m + 1)$.  This proves (i).

%*********************************************

\textbf{ii)  $\ol{N} \cong D_{2m}$:}  In this case,
\[\phi(x)=x^m+\frac{1}{x^m}= \frac {x^{2m} +1 } {x^m},\] 
which has branch points $q_1 = \infty$, $q_2 = 2$, and $q_3 = -2$.
Hence $\gamma_1(x) = x$, $\gamma_2(x) = x^m - 1$, and $\gamma_3(x) =
x^m + 1$.
The involution in the dihedral group permutes the branch points $q_2$ and
$q_3$, so $q_2 \in T$ if and only if $q_3 \in T$.  But if neither is
in $T$, then \cref{Eroots} shows that $\X$ has equation $y^n = x$,
contradicting the assumption that $g(\X) \geq 2$.  So $T = \{q_1, q_2,
q_3\}$ or $T = \{q_2, q_3\}$.  From \cref{Eroots}, we have the two possible equations

\[ y^n = x^{2m} -1, \quad y^n = x (x^{2m} -1). \]
This proves (ii).
%*************************************

\textbf{iii)    $\ol{N} \cong A_4$:} In this case,
 \[
 \phi(x)=\frac{x^{12}-33x^8-33x^4+1}{x^2(x^4-1)^2},
\]
which has branch points $q_1=\infty$ of index $2$, $q_2=6i\sqrt{3}$ of
index $3$, and
$q_3=-6i\sqrt{3}$ of index $3$, where $i^2 = -1$.
Hence
\[
  t_4 := \gamma_1  = x(x^4-1), \quad \gamma_2 = x^4+2i\sqrt{3}x^2+1, \quad \gamma_3 = x^4-2i\sqrt{3}x^2+1,
\]
with $\infty$ in the fiber of $q_1$ as well.
The branch points $q_1=\infty$, $q_2=6i\sqrt{3}$, and
$q_3=-6i\sqrt{3}$ are the branch points of the covering $\pi : \X \to
\P^1$.  Let $s_4$ and $t_4$ be as in \cref{S4-pols} below.  If neither $q_2$ nor $q_3$ is in $T$, then \cref{Eroots} shows
that the equation of $\X$ is $y^n = t_4$.  Observe that
$\gamma_2\gamma_3 = s_4$.  So if both $q_2$ and $q_3$ are in $T$, then
\cref{Eroots} shows that the equation of $\X$ is either $y^n = s_4t_4$
or $y^n = s_4$.
In all cases, the reduced automorphism group is actually $S_4$, so we
may assume that exactly one of $q_2$ or $q_3$ is in $T$.  Since the
two choices are conjugate, we may assume $q_2 \in T$ but $q_3 \notin T$.

So $T = \{q_1, q_2\}$ or $\{q_2\}$.  By \cref{Eroots}, we have the two possible equations
\[
 y^n = (x^4 + 2i\sqrt{3}x^2 + 1), \quad y^n=  x (x^4-1) (x^4+2i\sqrt{3}x^2+1).
\]
The last   assertion of (iii) is true because $\infty \in \phi^{-1}(q_1)$.

%**************************************************

\textbf{iv)  $\ol{N} \cong S_4$:} In this case,
\[
\phi (x) = \frac  {(x^8+14x^4+1)^3} {108 \, x^4 (x^4-1)^4},
\]
which has branch points $q_1=1$, $q_2=0$ and $q_3=\infty$.  Also, 
$\infty \in \phi^{-1}(q_3)$. Then
\begin{equation}\label{S4-pols}
\begin{split}
r_4(x)&: \gamma_1(x) = x^{12}-33x^8-33x^4+1\\
s_4(x)&:= \gamma_2(x) = x^8 + 14x^4 + 1\\
t_4(x)&:= \gamma_3(x) = x(x^4 - 1), \\
\end{split}
\end{equation}
Every possible case occurs here.  So by \cref{Eroots}, the equation of the curve
$\X$ is $y^n=f(x)$ for $f(x) $ one of
\[
r_4(x), s_4 (x), t_4(x), r_4(x) s_4(x),  r_4(x) t_4(x),  s_4(x) t_4 (x), r_4(x) s_4(x) t_4 (x). 
\]
The last  assertion of (iv) is true because $\infty \in \phi^{-1}(q_3)$.
%****************************************************

\textbf{v)  $\ol{N} \cong A_5$:} In this case,
\[
  \phi(x) = \frac{(-x^{20} + 228x^{15} - 494^{10} -228x^5 -
    1)^3}{(x(x^{10}+11x^5-1))^5},
\]
which has branch points of $q_1 = 0$, $q_2 = 1728$, and $q_3 =
\infty$.
One computes
\begin{equation}\label{A5-pols}
\begin{split}
r_5(x)&: \gamma_1(x) = =x^{20}-228x^{15}+494x^{10}+228x^5+1\\
s_5(x)&:= \gamma_2(x) = x(x^{10}+11x^5-1)\\
t_5(x)&:= \gamma_3(x) = x^{30}+522x^{25}-10005x^{20}-10005x^{10}-522x^5+1, \\
\end{split}
\end{equation}
and one notes that $\infty \in \phi^{-1}(q_2)$ as well.    Every
possible case occurs here.  So by \cref{Eroots}, the equation of the curve $\X$ is
$y^n=f(x)$ for $f(x) $ one of
\[
r_5(x), s_5 (x), t_5(x), r_5(x) s_5(x),  r_5(x) t_5(x),  s_5(x) t_5 (x), r_5(x) s_5(x) t_5 (x). 
\]
The last assertion of (v) is true because $\infty \in \phi^{-1}(q_2)$.
\end{proof}

\begin{rem}\label{Rpermute}
It is clear from the proof of \cref{prop-1} that in all cases, the
group $\ol{N}$ permutes the branch locus of $\pi: \X \to \proj^1_x$.
\end{rem}

\begin{prop}\label{Pnvalues}
\begin{enumerate}[\upshape (i)]
\item  Let $\X$ be a pre-superelliptic curve satisfying the conditions
  of \cref{prop-1} given by an affine equation
$y^n= f(x)$.  Suppose $\ol{N} \in \{A_4, S_4, A_5\}$ as in \cref{prop-1}(iii), (iv), or (v).  Then $n \mid (\deg f(x) + \delta)$, where
\[
\delta = \begin{cases} 1 & t_4(x) \mid f(x) \text{ or } s_5(x) \mid f(x) \\
    0 & \text{otherwise.} \end{cases}
\]
Furthermore, $\deg(f) + \delta$ is the number of branch points of the map $\X \to \proj^1$ given by projection to the $x$-coordinate.
\item Conversely, if $\X$ is a smooth projective curve given by an
  affine equation $y^n = f(x)$ as in \cref{prop-1}(iii), (iv), or (v)
  with $n \mid (\deg (f(x) + \delta)$, then $\X$ satisfies the assumptions of
  \cref{prop-1} (in particular, $\X \lar \X/N$ is branched at three
  points, with $N$ as in \cref{prop-1}).
\end{enumerate}

\end{prop}

\begin{proof}
We first note that by \cref{prop-1}(iii), (iv), and (v), the orbit of $\infty$ under $\ol{\Aut(\X)}$ consists of the roots of $t_4(x)$ (in cases (iii) and (iv)) or of $s_5(x)$ (in case (v)).  So $\infty$ is a branch point of $\X \to \proj^1$ if and only if $t_4(x) \mid f(x)$ or $s_5(x) \mid f(x)$; that is, if and only if $\delta = 1$.  

Since $f(x)$ is separable, the monodromy action induced by a small counterclockwise loop around any root of $f(x)$ takes a point $(a,b)$ to $(a, e^{2\pi i/n}b)$.  In particular, if $\tau \in \Aut_{\proj^1}(\X)$ is the isomorphism $(a, b) \mapsto (a, e^{2\pi   i/n}b)$, then $\tau$ has order $n$ and the signature of $\X \to \proj^1$ is
\[
  \begin{cases} (\tau, \ldots, \tau) & \delta = 0 \\
    (\tau, \ldots, \tau, \sigma) & \delta = 1,
  \end{cases}
\]
for some $\sigma \in \ints/n$ corresponding to the branch point
$\infty$.  Now, the product of all entries in the signature must be
the identity.  In the first case, this implies that the number of
branch points is divisible by $n$, so $\deg f(x)$ is divisible by $n$.
In the second case, since $\infty$ is permuted with other branch
points of the cover and $\ints/n$ is central in $\Aut(\X)$, we have
that $\sigma = \tau$.  Thus the number of branch points, which is now
$\deg f(x) + 1$, is also divisible by $n$.  This completes the proof
of part (i).

Let $\pi: \X \to \proj^1$ be the projection to the $x$-coordinate.  To
prove (ii), it suffices to show that if $\ol{\alpha} \in
\Aut(\proj^1)$ preserves the branch
locus of $\pi$, then $\ol{\alpha}$ lifts to an
element $\alpha$ of $\Aut(X)$.  By the proof of part (i), the
signature of $\pi$ is $(\tau, \ldots, \tau)$ for some $\tau \in
\Aut(\X/\proj^1)$.  After a change of variables, we may assume
that $\infty$ is not a branch point, so the affine equation is $y^n =
f(x)$ where $f(x)$ is separable.  Now, $\ol{\alpha}$ permutes the roots
of $f(x)$, since $\infty$ is not a branch point.  Thus it is clear that $\ol{\alpha}$
lifts to an automorphism of $\X$ given by acting trivially on $y$.  This
proves (ii).
\end{proof}

The following corollary is immediate.
\begin{cor}\label{Cuniformsignature}
If $\X$ is a superelliptic curve with many automorphisms and $H$ is a
superelliptic group, then the cover $\X \to \X/H$ has signature
$(\tau, \ldots, \tau)$ for some superelliptic automorphism $\tau$. 
\end{cor}

Let $\X$ be a pre-superelliptic curve with pre-superelliptic group $H$
with normalizer $N$ in $\Aut(\X)$.  Suppose $X \to X/N$ is branched at
exactly three points.  As a consequence of \cref{prop-1} and
\cref{Pnvalues}, there are finitely many such curves with reduced
automorphism group $A_4$, $S_4$, or $A_5$, and the list of such curves
is exactly \cref{Tcurvelist2} below.  In particular, all
superelliptic curves with many automorphisms and reduced automorphism
group not isomorphic to $C_m$ or $D_{2m}$ appear in \cref{Tcurvelist2}.

For the rest of the paper, our goal is to determine which superelliptic curves  with many automorphisms are CM-type. 

\begin{longtable}{|l|l|l|l|l|l|l|l|}%[hbt]
\caption{Curves with exceptional reduced automorphism  groups.}  %, equation $y^n = f(x)$
\label{Tcurvelist2}\\
%\begin{tabular}{|l|l|l|l|l|l|l|}
\hline
Nr. & $\ol{N}$   &$n$  & $g$ & $f(x)$  & CM? & Justification \\
\hline \hline
%
%\hline \hline  
  $\X_0$ & $A_4$ & 4 & 3 & $t_4$ & YES & \cref{Pstreitexamples}\\
           \hline
  $\X_1$ &      \multirow{3}{*}{$A_4$}                       &  2   &
                                                                      4
                             & \multirow{3}{*}{$p_4t_4$}& YES & \cref{Pstreitexamples}\\
$\X_2$ &                 &        5  &    16  & & NO & \cref{Pruleout}\\
$\X_3$ &                 &        10   &  36 & & NO &
                                                      \cref{Cmullerpink} \\
  \hline
$\X_4$ &      \multirow{5}{*}{$S_4$}                  &        2  & 5
                             &  \multirow{5}{*}{$r_4$}& YES & \cref{Pstreitexamples}\\
$\X_5$ &                 &        3   & 10  &  & NO& \cref{Pruleout}\\
$\X_6$ &                 &        4   &  15 &  & NO& \cref{Pruleout}\\
$\X_7$ &                 &        6   &  25 &  & NO& \cref{Cmullerpink}\\
$\X_8$ &                 &        12   & 55  & & NO & \cref{Cmullerpink}\\
\hline
$\X_9$ &      \multirow{3}{*}{$S_4$}                  &        2   &  3 &  \multirow{3}{*}{$s_4$}& NO& \cref{Phyperelliptic}\\
$\X_{10}$ &                    &       4   &  9 &  & NO &\cref{Phyperelliptic}\\
$\X_{11}$ &                  &        8   & 21  &  & NO &\cref{Phyperelliptic}\\
\hline
$\X_{12}$ &      \multirow{3}{*}{$S_4$}                  &        2   &  2 & \multirow{3}{*}{$t_4$}& YES &\cref{Pstreitexamples}\\
$\X_{13}$ &                    &        3   & 4  &  & YES & \cref{Pstreitexamples}\\
$\X_{14}$ &                  &        6  &  10  &  & YES & \cref{Pstreitexamples}\\
\hline
$\X_{15}$ &      \multirow{5}{*}{$S_4$}                  &        2   & 9  & \multirow{5}{*}{$r_4s_4$}& NO &\cref{Phyperelliptic}\\
$\X_{16}$ &                   &        4   &  27 &  & NO &\cref{Phyperelliptic}\\
$\X_{17}$ &                   &        5   &  36 &  & NO &\cref{PBadReductionCurves}
\\
$\X_{18}$ &                   &        10   & 81  & & NO &\cref{Phyperelliptic}\\
$\X_{19}$ &                   &        20   & 171  & & NO &\cref{Phyperelliptic}\\
\hline
$\X_{20}$ &      \multirow{5}{*}{$S_4$}                  &        2   &  8 & \multirow{5}{*}{$r_4t_4$}& YES &\cref{Pstreitexamples}\\
$\X_{21}$ &                   &        3   & 16  &  & NO &\cref{PBadReductionCurves}
\\
$\X_{22}$ &                   &        6   &  40 & & NO &\cref{Ctree}\\
  $\X_{23}$ &                  &        9   &  64 & & NO &\cref{Ctree}\\
 $\X_{24}$ &                   &        18   & 136  & & NO &\cref{Ctree}\\
  \hline
$\X_{25}$ &      \multirow{3}{*}{$S_4$}                  &        2   & 6  &  \multirow{3}{*}{$s_4t_4$}& NO &\cref{Phyperelliptic}\\
$\X_{26}$ &                   &        7   &  36 &  & NO &\cref{Pruleout}\\
$\X_{27}$ &                   &        14   & 78  & & NO &\cref{Phyperelliptic}\\
  \hline
$\X_{28}$ &      \multirow{3}{*}{$S_4$}                  &        2   &  12 &  \multirow{3}{*}{$r_4s_4t_4$}& NO &\cref{Phyperelliptic}\\
$\X_{29}$ &                   &        13   & 144  &  & NO &\cref{Pruleout}\\
$\X_{30}$ &                   &        26   &  300 &  & NO &\cref{Phyperelliptic}\\
  \hline
$\X_{31}$ &      \multirow{5}{*}{$A_5$}                  &        2   & 9  &  \multirow{5}{*}{$r_5$}& NO &\cref{Phyperelliptic}\\
$\X_{32}$ &                   &        4   &  27 &  & NO &\cref{Phyperelliptic}\\
$\X_{33}$ &                   &        5   &  36 &  & NO
        &\cref{Pruleout} \\
$\X_{34}$ &                   &        10   & 81  &  & NO &\cref{Phyperelliptic}\\
$\X_{35}$ &                   &        20   &  171 & & NO &\cref{Phyperelliptic}\\
  \hline
$\X_{36}$ &      \multirow{5}{*}{$A_5$}                  &        2   & 5  &  \multirow{5}{*}{$s_5$}& NO &\cref{Phyperelliptic}\\
$\X_{37}$ &                   &        3   &  10 &  & YES &\cref{Pstreitexamples}\\
$\X_{38}$ &                   &        4   &  15 &  & NO &\cref{Phyperelliptic}\\
$\X_{39}$ &                   &        6   &  25 &  & NO &\cref{Phyperelliptic}\\
$\X_{40}$ &                   &        12   & 55  & & NO &\cref{Phyperelliptic}\\
  \hline
$\X_{41}$ &      \multirow{7}{*}{$A_5$}                  &        2   & 14  &  \multirow{7}{*}{$t_5$}& YES &\cref{Pstreitexamples}\\
$\X_{42}$ &                   &        3   &  28 &  & NO &\cref{PBadReductionCurves}
\\
$\X_{43}$ &                   &        5   &  56 &  & NO & \cref{Pruleout}\\
$\X_{44}$ &                   &        6   &  70 &  & NO &\cref{Ctree}\\
$\X_{45}$ &                   &        10   & 126  & & NO & \cref{Cmullerpink}\\
$\X_{46}$ &                   &        15   &  196 &  & NO &\cref{Ctree}\\
$\X_{47}$ &                   &        30   &  406 & & NO &\cref{Ctree}\\
  \hline
$\X_{48}$ &      \multirow{5}{*}{$A_5$}                  &        2   & 15  &  \multirow{5}{*}{$r_5s_5$}& NO &\cref{Phyperelliptic}\\
$\X_{49}$ &                   &        4   & 45  &  & NO &\cref{Phyperelliptic}\\
$\X_{50}$ &                   &        8   &  105 &  & NO &\cref{Phyperelliptic}\\
$\X_{51}$ &                   &        16   & 225  &  & NO &\cref{Phyperelliptic}\\
$\X_{52}$ &                   &        32   & 465  & & NO &\cref{Phyperelliptic}\\
  \hline
$\X_{53}$ &      \multirow{5}{*}{$A_5$}                  &        2   & 24  &  \multirow{5}{*}{$r_5t_5$}& NO &\cref{Phyperelliptic}\\
$\X_{54}$ &                   &        5   & 96  &  & NO &\cref{PBadReductionCurves}
\\
$\X_{55}$ &                   &        10   & 216  &  & NO &\cref{Phyperelliptic}\\
$\X_{56}$ &                   &        25   &  576 &  & NO &\cref{Ctree}\\
$\X_{57}$ &                   &        50   &  1176 & & NO &\cref{Phyperelliptic}\\
  \hline
$\X_{58}$ &      \multirow{7}{*}{$A_5$}                  &        2   & 20  &  \multirow{7}{*}{$s_5t_5$}& NO &\cref{Phyperelliptic}\\
$\X_{59}$ &                   &        3   &  40 &  & NO &\cref{PBadReductionCurves}
\\
$\X_{60}$ &                   &        6   &  100 &  & NO &\cref{Phyperelliptic}\\
$\X_{61}$ &                   &        7   &  120 &  & NO &\cref{PBadReductionCurves}
\\
$\X_{62}$ &                   &        14   & 260  & & NO &\cref{Phyperelliptic}\\
$\X_{63}$ &                   &        21   &  400 &  & NO &\cref{Ctree}\\
$\X_{64}$ &                   &       42   &  820 & & NO &\cref{Phyperelliptic}\\
  \hline
$\X_{65}$ &      \multirow{3}{*}{$A_5$}                  &       2   &  30 &  \multirow{3}{*}{$r_5s_5t_5$}& NO &\cref{Phyperelliptic}\\
$\X_{66}$ &                   &        31   &  900 &  & NO
        &\cref{Pruleout} \\
$\X_{67}$ &                   &        62   &  1830 &  & NO &\cref{Phyperelliptic}\\
  \hline
\end{longtable}
%****************************************************************
\section{Positive CM Results}\label{Spositive}
In this section, we confirm that all superelliptic curves of separable
type with many automorphisms and reduced automorphism group $C_m$ or
$D_{2m}$ have CM, and we show that the curves in Table
\cref{Tcurvelist2} marked ``YES'' have CM as well.

\subsection{Quotients of Fermat curves}\label{Sfermatquotient}
A \emph{Fermat curve} is a projective curve with affine equation $x^a + y^a + z^a
= 0$ for some $a \in \nats$.   It is easy to show directly that any superelliptic curve $\X$  with
many automorphisms and reduced automorphism group $\ol{G} = C_m$ or
$\ol{G} = D_{2m}$ is
isomorphic to a quotient of a Fermat curve, and thus has CM:

\begin{lem}\label{Lfermat}
If an algebraic curve  $\X$ is the quotient of a Fermat curve then $\Jac (\X)$  has CM. 
\end{lem}

\begin{proof}
It is well-known that Fermat curves have CM Jacobians, see, e.g., \cite[Ch.\
VI.1]{Schmidt}.  The lemma follows since any quotient of a curve with
CM Jacobian has CM Jacobian \cref{CCMquotient}.
\end{proof}

\begin{thm}\label{Cstreit}
Let $\X$ be a superelliptic curve  with many automorphisms. 

i)  If $\ol{N}$ is cyclic   then $\Jac (\X)$ has CM.

ii)  If $\ol{N}$ is   dihedral then $\Jac (\X)$ has CM. 

\end{thm}

\begin{proof}
    First note that for any $a \in \mathbb{N}$, the smooth proper curve with affine equation $x^a = y^a \pm 1$ is isomorphic
to a Fermat curve over $\complex$.

If $\ol{G} \iso C_m$ or $\ol{G} \iso D_{2m}$, then by \cref{prop-1}(i) and (ii), $\X$ has (affine) equation $y^n = x^r \pm 1$ or $y^n = x(x^r \pm 1)$ for some $n$ and $r$.  In the
first case, $\X$ is clearly a quotient of the Fermat curve $\mc{Y}$ with affine equation $u^{rn} = v^{rn} \pm 1$  under the automorphism group generated by $u \mapsto \zeta_m u$ and $v  \mapsto \zeta_n v$.  In the second case, the quotient of the Fermat curve with affine equation $u^{rn} = v^{rn} \pm 1$ by the automorphism group generated by $(u,v) \mapsto (\zeta_{rn}^{-1}u, \zeta_n v)$ and
 $(u,v) \mapsto (\zeta_r u, v)$ is  $\X$, as we see by  setting $x = v^n$ and $y = u^rv$.  By \cref{Lfermat}, $\X$ has CM, proving part (i).
\end{proof}

\begin{rem}
In \cref{PCmD2m} below, we give another proof of \cref{Cstreit}.
\end{rem}

\subsection{Streit's criterion}\label{Sstreit}
Let $\X/\complex$ be a superelliptic curve , with $\sigma \in \Aut (\X)$ and  $\bar \sigma$  its image in $\ol{N}$. Let $\chi_\X(\sigma)$ be the character of $\sigma$  on $H^0(\X, \omega_{\X/\complex})$, and let $\Sym^2\chi_\X(\sigma)$ be the character of $\sigma$ on the $\Aut(\X)$-representation $\Sym^2H^0(\X, \omega_{\X/\complex})$.

\begin{lem}[Streit~\cite{streit-01}] \label{Streit}
$\Jac (\X)$ has CM if $\langle \Sym^2\chi_{\X}, \chi_{triv} \rangle = 0$.  
\end{lem}

\begin{rem}\label{Ranygroup}
If $H$ is any subgroup of $\Aut(\X)$, it suffices to verify Streit's criterion considering $\Sym^2H^0(\X, \omega_{\X/\complex})$ as an $H$-representation.  This is because $\langle \Sym^2\chi_{\X}, \chi_{triv} \rangle_H = 0$ implies that $\langle \Sym^2\chi_{\X}, \chi_{triv} \rangle_G = 0$, since the former means that $\Sym^2H^0(\X,
\omega_{\X/\complex})$ has no $H$-invariant vectors whereas the latter means it has no $G$-invariant vectors.
\end{rem}

To calculate $\Sym^2\chi_{\X}$, we record the following lemma.
\begin{lem}\label{Lsymsquare}
We have
\[ 
\sum_{\sigma \mapsto \ol{\sigma}} \Sym^2\chi_\X(\sigma) =  \frac{1}{2} \left(\sum_{\sigma \mapsto \ol{\sigma}} \chi_\X(\sigma^2) +    \sum_{\sigma \mapsto \ol{\sigma}} \chi_\X(\sigma)^2 \right).
\]
\end{lem}
 
\begin{proof}
  This is a basic result of representation theory.
\end{proof}

We now have the following:

\begin{prop}\label{Pstreit}
Let $\X:  y^n = f(x)$ be a smooth superelliptic curve  defined over $\complex$.  
For each $\ol{\sigma} \in \ol{\Aut} (\X)$, let $m$ be its order, and
let $\zeta_{\ol{\sigma}}$ be either
ratio of the eigenvalues when  $\ol{\sigma}$ is thought of as an element of
$PGL_2(\complex)$ ($\zeta_{\ol{\sigma}}$ is a primitive $m$th root of unity). Define
\[
k_{\ol{\sigma}} =
\left\{
\begin{split}
& 1     \text{  if   }    \ol{\sigma}  \; \text{ fixes a branch point
  of   }  \, \X \to  \proj^1 \\
 &  0 \quad \text{ otherwise.}  \\
\end{split} 
\right.
\]
Let $\zeta_n$, be a primitive $n$th root of unity, and let $\zeta_{n,
  \ol{\sigma}}$ be a primitive $mn$th root of unity such that
$\zeta_{n, \ol{\sigma}}^n = \zeta_{\ol{\sigma}}$.  Let $A$ be the set
of ordered pairs defined below in \cref{Eholforms}.  If
\[
\sum_{\ol{\sigma} \in \ol{\Aut} (\X)} \sum_{i=0}^{n-1}    \left(
  \sum_{(a,b) \in A} \zeta_{\ol{\sigma}}^{2(a+1)} \zeta_n^{2(b+1)i}
  \zeta_{n, \ol{\sigma}}^{2(b-n+1)k_{\ol{\sigma}}} +
  \left(\sum_{(a,b) \in A}   \zeta_{\ol{\sigma}}^{a+1}  \zeta_n^{(b+1)i}
     \zeta_{n, \ol{\sigma}}^{(b-n+1)k_{\ol{\sigma}}}
  \right)^2 \right)
\]
vanishes, then $\X$ has CM.
\end{prop}

\proof 
Suppose $\X$ has
genus $g$ and $\deg(f) = d$.  From \cite{towse}, 
a basis for the space of holomorphic differentials on $\X$ is given by
\[
  x^a y^b    \left( \frac{dx}{y^{n-1}} \right),
\]
as $(a,b)$ ranges through the set
\begin{equation}\label{Eholforms}
A := \{(a, b) \in \ints^2 \mid a \geq 0, 0 \leq b < n, 0 \leq   an+bd
\leq 2g-2\}.
\end{equation}
By the Hurwitz formula, $2g - 2 = -n - \gcd(n,d) + d(n-1)$, so we can also write $A$ as
\begin{equation}\label{Edifferentialbasis}
  \{(a, b) \in \ints^2 \mid  0 \leq b < n, 0 \leq
    a \leq d-1 - \frac{d(1+b) + \gcd(n,d)}{n}\}.
\end{equation}
Let $\sigma \in \Aut \X$,  $\ol{\sigma}$ its image in  $\ol{\Aut} (\X)$, and  $m$ be the order of $\ol{\sigma}$. 
If $\tau$ is the superelliptic automorphism of $\X$ and  $H  = \langle
\tau \rangle$, then $\sigma H$ is a coset consisting of $n$ different
automorphisms of $\X$, say $\sigma = \sigma_1, \sigma_2,\ldots,
\sigma_n$, all projecting to  $\ol{\sigma} \in \ol{\Aut} (\X) \iso
\Aut(\proj^1)$. Now, $\ol{\sigma}$ acts on $\proj^1$ with two fixed
points, and after a change of coordinate we  may assume that they are
$0$ and $\infty$.  After possibly replacing $x$ by $1/x$, we may assume that $\ol{\sigma}$ acts on the
coordinate $x$ via $x \mapsto \zeta_{\ol{\sigma}} x$.  After this
change of variables, there is a polynomial $h \in k[x]$ such that the
equation for $\X$ is given by $y^n = x^{k_{\ol{\sigma}}}   h(x^m)$,
where $k_{\ol{\sigma}} = 1$ if the fixed point $0$ of $\ol{\sigma}$ is
a ramification point of $X \to \proj^1$, and $k_{\ol{\sigma}} = 0$
otherwise, as in the statement of the proposition.  

Fix a primitive $n$-th root of unity $\zeta_n$, as well as a
primitive $mn$-th root $\zeta_{n, \ol{\sigma}}$ as in the statement of
the proposition, so $\zeta_{n, \ol{\sigma}}^n  = \zeta_m$.  Since $\ol{\sigma}(x) = \zeta_{\ol{\sigma} }x$ and $y^n = x^{k_{\ol{\sigma}} }h(x^m)$, we  have that 
\[ 
\sigma (y) = \zeta_n^i \zeta_{n, \ol{\sigma}}^{k_{\ol{\sigma} }} y,
\]
for some $i \in \{1,  \ldots, n\}$.  After reordering $\sigma_1,
\ldots, \sigma_n$, we may  assume that $\sigma_i(y) = \zeta_n^i
\zeta_{n, \ol{\sigma}}^{k_{\ol{\sigma}} } y$ for each $i$.  In  particular,  $x^a y^b \frac {dx} {y^{n-1}}$ is an eigenvector for every $\sigma_i$.  Its eigenvalue is
\[
\zeta_{\ol{\sigma}}^{a +    1}\zeta_n^{i(b-n+1)}\zeta_{n,
  \ol{\sigma}}^{(b-n+1)k_{\ol{\sigma}} } = \zeta_{\ol{\sigma}}^{a +
  1}\zeta_n^{i(b+1)}\zeta_{n, \ol{\sigma}}^{(b-n+1)k_{\ol{\sigma} }}
\]    
We thus have
\begin{equation}\label{Echi}
 \sum_{\sigma \mapsto \ol{\sigma}} \chi_X(\sigma) = \sum_{i=0}^{n-1}
 \sum_{(a,b) \in A} \zeta_{\ol{\sigma} }^{a+1} \zeta_n^{(b+1)i}
 \zeta_{n, \ol{\sigma}}^{(b-n+1)k_{\ol{\sigma} }}.
\end{equation}
Likewise,
  \begin{equation}\label{Echisigma2}
    \sum_{\sigma \mapsto \ol{\sigma}} \chi_X(\sigma^2) = \sum_{i=0}^{n-1}
    \sum_{(a,b) \in A} \zeta_{\ol{\sigma}}^{2(a+1)} \zeta_n^{2(b+1)i}
    \zeta_{n, \ol{\sigma}}^{2(b-n+1)k_{\ol{\sigma} }}.
  \end{equation}
Also,
\begin{equation}\label{Echi2sigma}
\sum_{\sigma \mapsto \ol{\sigma}} \chi_X(\sigma)^2 = \sum_{i=0}^{n-1}
\left(\sum_{(a,b) \in A}   \zeta_{\ol{\sigma} }^{a+1}
  \zeta_n^{(b+1)i}        \zeta_{n, \ol{\sigma}}^{(b-n+1)k_{\ol{\sigma}}}\right)^2.
\end{equation} 
By \cref{Lsymsquare}, $\Sym^2\chi_\X(\sigma)  = \frac{1}{2}(\chi_\X(\sigma^2) + \chi_\X(\sigma)^2)$.  Combining this with \cref{Echisigma2} and \cref{Echi2sigma}, we have 
\[ 
\begin{split}
\sum_{ \sigma \mapsto \ol {\sigma}}   \Sym^2 \chi_\X(\sigma)   =
\frac{1}{2}    &     \left( \sum_{i=0}^{n-1}    \left( \sum_{(a,b) \in
      A} \zeta_{\ol{\sigma}}^{2(a+1)} \zeta_n^{2(b+1)i}    \zeta_{n, \ol{\sigma}}^{2(b-n+1)k_{\ol{\sigma}}}  +  \right. \right. \\
 & \left. \left. +       \left(\sum_{(a,b) \in A}
       \zeta_{\ol{\sigma}}^{a+1}  \zeta_n^{(b+1)i}        \zeta_{n, \ol{\sigma}}^{(b-n+1)k_{\ol{\sigma}}}\right)^2 \right)\right).  \\
\end{split} 
\]
Combining the above with \cref{Streit} we claim the result. 

\qed

\begin{prop}\label{Pstreitexamples}
The curves $\X_0$, $\X_1$, $\X_4$, $\X_{12}$, $\X_{13}$, $\X_{14}$, $\X_{20}$,
$\X_{37}$, $\X_{41}$ all have CM.
\end{prop}

\begin{proof}
The GAP program \texttt{streit\_program.gap}\footnote{available as
  part of the arxiv posting for this paper} computes the sum in
\cref{Pstreit} for any superelliptic curve , presented as in \cref{prop-1}, with $n$, $\ol{\Aut} (\X)$, and $f(x)$ as inputs.  The program is modelled on that of Pink and M\"{u}ller used in \cite{muller-pink}.  To calculate $k_{\ol{\sigma}}$, we use the embedding of $\ol{G} = \ol{\Aut} (\X)$ into $PGL_2(\complex)$ from
\cref{prop-1} and its proof.  The rest of the calculation is straightforward.  For all of the curves in the proposition, the sum  in \cref{Pstreit} comes to $0$.
\end{proof}

\begin{rem}
The curves $\X_1$, $\X_4$, $\X_{12}$, $\X_{20}$, and $\X_{41}$ are  all hyperelliptic, and Streit's criterion was already verified for  them in \cite{muller-pink}.  These correspond to $X_4$, $X_7$,  $X_5$, $X_9$, and $X_{14}$ respectively in that paper.
\end{rem}
%Our GAP program
% \andrew{(put program on a website and cite it --- also  give Pink and M\"uller credit for the framework.)}    computes that the expression in Proposition~\cref{Pstreit} is $0$ in all cases, except for curves 24 and 36.  
%\qed

\begin{prop}\label{PCmD2m}
Let $\X$ be a superelliptic curve  with many automorphisms, and assume that $\ol{\Aut(\X)} = C_m$ or $D_{2m}$.  Then $\X$ satisfies Streit's criterion.  That is, $\langle \Sym^2_{\chi_{\X}}, \chi_{triv} \rangle = 0$.
\end{prop}

\begin{proof}
Let $V = \Sym^2(H^0(\X, \omega_{\X/\complex}))$, and let $G = \Aut(\X)$.  It suffices to show that no non-trivial points of $V$ are fixed by $\Aut(\X)$.

Let us first assume that $\ol{\Aut(\X)} = C_m$.  By \cref{prop-1}, the affine
equation of $\X$ is $y^n = x^k(x^m + 1)$, where $k \in \{0,1\}$.  The set
$\{v_{a, b} := x^ay^bdx/y^{n-1} \mid (a, b) \in A\}$ is a basis of simultaneous eigenvectors for the
action of $G$ on $H^0(\X, \omega_{\X/\complex})$, where $A$ is as in \cref{Eholforms}.  For
each $v_{a,b}$ in this basis, let $\lambda_{a,b}(g)$ be its eigenvalue
under the action of $g \in G$.  It suffices to prove that there is no
set $\{(a_1, b_1), (a_2, b_2)\}$ of indices such that $\lambda_{a_1,
  b_1}(g)\lambda_{a_2, b_2}(g)$ takes the
constant value $1$ on $G$.

Suppose $\sigma(x,y) = (\zeta_mx,\zeta_{mn}^ky)$, and  $\tau(x,y) = (x,
\zeta_n y)$, where $\zeta_{mn}$ is a primitive $mn$th root
of unity with $\zeta_{mn}^n = \zeta_m$.  The elements $\sigma$ and $\tau$ generate $G$.  Then
$$\lambda_{a_1,b_1}(\sigma^i\tau^j) \lambda_{a_2,b_2}(\sigma^i\tau^j) = \zeta_m^{i(a_1+a_2 + 2)}\zeta_n^{j(b_1 + b_2 + 2 - 2n)}\zeta_{mn}^{ik(b_1 + b_2 + 2 - 2n)}.$$
This is independent of $j$ only if $b_1 + b_2 + 2 = n$, in which case
it equals
$$\zeta_m^{i(a_1+a_2 + 2 - k)}.$$
By \cref{Edifferentialbasis}, we have
\begin{equation}\label{Eabound}
a_1 + a_2 + 2 \leq 2d  - \frac{d(b_1 + b_2 + 2) + \gcd(n, d)}{n} =
d - \frac{\gcd(n,d)}{n},
\end{equation}
where $d = m + k$.  So $0 < a_1 + a_2 + 2 - k < m$, which means that
$\zeta_m^{i(a_1+a_2 + 2 - k)}$ is not
independent of $i$.  Thus the eigenvalue cannot take the constant value $1$ on
$G$.

Now, assume $\ol{\Aut(\X)} = D_{2m}$.  By \cref{prop-1}, the affine equation
of $\X$ is $y^n = x^k(x^{2m}-1)$, where $k \in \{0,1\}$.  The
automorphism group $G$ of $X$ is generated by 
$$\sigma(x,y) = (\zeta_mx, \zeta_{mn}^ky),\ \  \tau(x,y) = (x,
\zeta_ny),\ \ \rho(x,y) = (1/x, \zeta_{2n}y/x^{2(m+k)/n}),$$ where
$\zeta_{2n}$ is any $2n$th root of unity and $\zeta_m$, $\zeta_n$, and
$\zeta_{mn}$ are as before.  Let $H$ be
the index two subgroup of $G$ generated by $\sigma$ and $\tau$.   
The same $v_{a,b}$ as in the $C_m$ case form a basis of simultaneous
eigenvectors for the action of $H$ on $H^0(\X,
\omega_{\X/\complex})$.  If we again set $\lambda_{a,b}$ to be the
respective eigenvalues, then exactly as in the $C_m$ case,
$\lambda_{a_1, b_1}(h)\lambda_{a_2,b_2}(h)$ takes the constant value
$1$ as $h$ ranges over $H$ only if $b_1 + b_2 + 2 = n$ and $a_1 + a_2
+ 2 - k$ is divisible by $m$.  Furthermore, since $d = 2m + k$, we
know from \cref{Eabound} that $a_1 + a_2 + 2 - k$ is divisible by $m$ only if $a_1 + a_2
+2 - k = m$.

The only eigenvector with eigenvalue $1$ for the action of $H$ on $\Sym^2 (H^0 (\X, \omega_{\X/\complex}))$ is
\begin{align*}
  \omega := & x^{a_1+a_2}y^{b_1 +b_2 - 2n           +2}(dx)^2          = x^{m+k-2}y^{-n}(dx)^2 \\
         &= \frac{x^{m+k-2}}{x^k(x^{2m}-1)} (dx)^2 = \frac{x^m}{x^{2m}-1}\left(\frac{dx}{x}\right)^2.
\end{align*}
One sees immediately that $\rho(\omega) = - \omega$, so $\omega$ is
not fixed under $G$, which completes the proof.
\end{proof}

Note that \cref{PCmD2m} gives another proof of \cref{Cstreit}.

%*******************************************************
\section{Negative CM results}\label{Snegative}

In this section, we show that the remaining curves in \cref{Tcurvelist2} do not have CM.

\subsection{Bootstrapping the hyperelliptic  case}\label{Shyperelliptic}
The following proposition is a direct consequence of the main result of \cite{muller-pink}.

\begin{prop}\label{Phyperelliptic}
None of the curves $\X_i$ in \cref{Tcurvelist2} for
\begin{equation*}
    \begin{split}
      i \in \{9, 10, 11, 15, 16, 18, 19, 25, 27, 28, 30, 31, 32, 34, 35, 36, 38,\\ 39, 40, 48, 49, 50, 51, 52, 53, 55, 57, 58, 60, 62, 64, 65, 67\}
\end{split}
\end{equation*}
has CM.
\end{prop}

\begin{proof}
The curves $\X_i$ for $i \in S := \{9, 15, 25, 28, 31, 36, 48, 53, 58,  65\}$ are all hyperelliptic, and were shown not to have CM in  \cite[Table 1]{muller-pink}.  For each of the other curves $\X_i$ in the proposition,  there exists $j \in S$ such that $\X_i$ has $\X_j$ as a quotient by  an automorphism fixing $x$ and multiplying $y$ by an appropriate  root of unity.  Since $\X_j$ does not have CM, neither does $\X_i$.
\end{proof}

%*********
\subsection{Using stable reduction}\label{Sstable}
It is well-known that if $\X$ is a CM curve defined over a number
field $K$, then its Jacobian $\Jac
\X$ has potentially good reduction modulo all primes of $K$
(\cite[Theorem 6]{Serre}).  For any such prime $\mf{p}$, let
$K_{\mf{p}}$ be the corresponding completion of $K$.
Assume the genus of $\X$ is at least $2$.
The \emph{stable reduction theorem} states that there exists a finite extension
$L/K_{\mf{p}}$ for which $\X \times_K L$ has a \emph{stable model}
$\X^{st}$ over $\Spec \mc{O}_L$, where $\mc{O}_L$ is the ring of integers of
$L$ (see, e.g., \cite[Corollary 2.7]{DM}).  Specifically, $\X^{st} \to
\Spec \mc{O}_L$ is a
flat relative curve whose generic fiber is
isomorphic to $\X$ and whose special fiber $\ol{\X}$ (called the
\emph{stable reduction} of $\X$ modulo $\mf{p}$) is
reduced, has smooth irreducible components, has only ordinary
double points for singularities, and has the property that each
irreducible component of genus zero contains at least three singular
points of $\ol{\X}$.
One forms the \emph{dual graph}
$\Gamma_{\ol{\X}}$ of $\ol{\X}$ by taking the vertices of $\Gamma_{\ol{\X}}$ to
correspond to the irreducible components of $\ol{\X}$, with an edge
between two vertices for each point where the two corresponding components
intersect.  A model $\X^{ss}$ of of $\X \times_K L$ is called a
\emph{semistable model} if it satisfies all the properties of a stable
model, except possibly the requirement on genus zero irreducible
components.  The dual graph of the special fiber of any $\X^{ss}$ is homeomorphic to
$\Gamma_{\ol{X}}$.

One has a similar construction for smooth curves $\mc{Y}/K$ with marked
points.  To wit, if $\lambda_1, \ldots, \lambda_m$ are points of
$\mc{Y}(K)$ and $2g + m \geq 3$, there exists a unique stable model
$\mc{Y}^{st}$ of the marked curve $(\mc{Y}, \{\lambda_1, \ldots,
\lambda_m\})$, which is defined as above, except that we require only
that each genus zero component of the special fiber $\ol{\mc{Y}}$
contain at least three points that are \emph{either} singular points of
$\ol{\mc{Y}}$ \emph{or} specializations of marked points. We also require
that the marked points specialize to distinct smooth points of $\ol{\mc{Y}}$.  Note that
if $\mc{Y} = \proj^1$, then $\Gamma_{\ol{\mc{Y}}}$ is always a tree.

By \cite[Chapter 9, \S2]{BLR}, $\Jac
\X$ has potentially good reduction if and only if $\Gamma_{\ol{\X}}$ is
a tree (i.e., has trivial first homology).  Thus, if we can find a
prime $\mf{p}$ for which $\Gamma_{\ol{\X}}$ is not a tree, then $\Jac
\X$ does not have CM.  We will use this criterion for 
several of the curves in \cref{Tcurvelist2}, generalizing
\cite[\S10]{MuellerThesis} to the superelliptic case.

For the rest of  \cref{Sstable}, let $K$ be a complete discrete valuation field with residue field $k$,
and let $n \in \nats$ with char$(k) \nmid n$.
Let $\X \to \proj^1_K$ be the (potentially) $\ints/n$-cover of
$\proj^1_K$ given by the affine equation $$y^n = \prod_i (x - \alpha_i),$$
where the $\alpha_i$ are pairwise distinct elements of $K$.  Let $B$ be the set of branch points of $\X \to \proj^1_K$ (so
$B$ consists of the $\alpha_i$, as well as $\infty$ if $n \nmid
\deg(\prod_i (x - \alpha_i)$).  Let $\mc{Y}^{st}$ be the stable model of the marked curve
$(\proj^1_K, B)$, and let $\Gamma_{\ol{\mc{Y}}}$ be the dual graph of
its special fiber.
Assume that $g(\X) \geq 2$, which in turn implies $|B| \geq 3$.

\begin{lem}\label{Lstableexplicit}
There exists a finite extension $L/K$ with valuation ring $\mc{O}_L$, such
that the normalization of $\mc{Y}^{st} \times_{\mc{O}_K} \mc{O}_L$ in
$L(\X)$ is a semistable model of $\X$ over $L$. 
\end{lem}

\begin{proof}
This follows from \cite[Corollary 3.6]{BW} and its proof, with $(\proj^1_K, B)$ playing
the role of $(X_{L_0}, D_{L_0})$.
\end{proof}

The normalization from \cref{Lstableexplicit} induces a map
\begin{equation}\label{Eprojection}
  \pi: \Gamma_{\ol{\X}} \to \Gamma_{\ol{\mc{Y}}}
\end{equation}
of graphs.

\begin{lem}\label{Lnonisom}
Let $\pi$ be as in \cref{Eprojection}.
  \begin{enumerate}[\upshape (i)]
    \item If $v$ is a leaf of $\Gamma_{\ol{\mc{Y}}}$ (i.e., a vertex
      incident to only one edge), then $|\pi^{-1}(v)| \leq n/2$.
    \item Suppose there exists an edge $e$ of 
$\Gamma_{\ol{\mc{Y}}}$ such that removing $e$ splits
$\Gamma_{\ol{\mc{Y}}}$ into two trees $T_1$ and $T_2$ where $n$ divides the number of elements of $B$
 specializing to each of $T_1$ and $T_2$.  Then $|\pi^{-1}(e)| = n$.
 \end{enumerate}
\end{lem}

\begin{proof}
  By the definition of marked stable model, the irreducible component
  $J$ of $\ol{\mc{Y}}$ corresponding to $v$ contains the specialization of
  at least one element $b$ of $B$.  Since the cover $\X \to \proj^1_K$ is
  potentially Galois and ramification indices are at least 2, the preimage of $b$ in $\X$ contains at
  most $n/2$ points.  Since every irreducible component of $\ol{\X}$
  lying above $J$ contains the specialization of one of these points,
  there are at most $n/2$ such irreducible components.  By the
  construction of $\pi$, this proves
  part (i).

  Now, assume $e$, $T_1$, and $T_2$ are as in part (ii).  Let $v$ be
  the unique vertex of $T_2$ incident to $e$, and let $J$ be the
  corresponding irreducible component of $\ol{\mc{Y}}$.  For $i \in
  \{1, 2\}$, write $B_i$ for the subset of $B$ consisting of branch
  points specializing to $T_i$.  Choose three distinct elements $\alpha$, $\beta$, and $\gamma$ of $B$, such
  that $\alpha \in B_1$ and such that $(\alpha, \beta, \gamma)$
  corresponds to $v$ via \cite[Proposition 4.2(3)]{BW}.  Note that in
  the construction of \cite{BW}, $\alpha$ can be chosen freely, and
  then it is automatic that $\beta, \gamma \in B_2$.  This is because
  if, say, $\beta \in B_1$, then in the language
  of \cite{BW}, we would have $\ol{\lambda}_t(\beta) = \ol{\lambda}_t(\alpha)$, which contradicts
  the definition of $\lambda_t$ above \cite[Proposition 4.2]{BW}). 
  The same holds for $\gamma$.

  As in \cite[Notation 4.4]{BW}, the triple $(\alpha, \beta, \gamma)$ gives rise
  to a coordinate $x_v$ on $\proj^1_K$ whose reduction $\ol{x}_v$
  is a coordinate on $J$ such that $\ol{x}_v = 0$ at the
  point of $J$ corresponding to the edge $e$.\footnote{Specifically, we have
  $x_v = \frac{\beta - \gamma}{\beta - \alpha}\frac{x - \alpha}{x-\gamma}$.}
 Since $\ol{x}_v$ is a coordinate on $J$, we in fact have that an
 element $b \in B$ satisfies $\ol{x}_v(b) = 0$ if and only if $b \in B_2$.

 As in \cite[\S4.3]{BW}, we can write $f$ in terms of the variable $x_v$, multiply by an
  appropriate element of $K$, and then reduce modulo a uniformizer of
  $K$ to obtain an element $\ol{f}_v \in k(\ol{x}_v)$.  Since $\ol{x}_v(b) = 0$
  if and only if $b \in B_2$, the order of $\ol{f}_v$ at
  $\ol{x}_v = 0$ is $|B_2|$, which is assumed to be
  divisible by $n$.

  By \cite[Proposition 4.5]{BW}, the restriction of the cover $\ol{\X}
  \to \ol{\mc{Y}}$ above $J$ is given birationally by the equation $y_v^n =
  \ol{f}_v$.  Since $\ord_{\ol{x}_v = 0}(\ol{f}_v)$ is divisible by
  $n$, the preimage of $\ol{x}_v = 0$ in $\ol{\X}$ has
  cardinality $n$.  This means that $|\pi^{-1}(e)| = n$, proving part (ii).
\end{proof}

\begin{prop}\label{Ptree}
In the situation of \cref{Lnonisom}(ii), the graph $\Gamma_{\ol{\X}}$ is
not a tree.
\end{prop}

\begin{proof}
Let $e$ be the edge from \cref{Lnonisom}(ii).  Consider the graph
$\Gamma$ constructed by removing $\pi^{-1}(e)$ from
$\Gamma_{\ol{\X}}$.  Then $\Gamma$ is the disjoint union of
$\pi^{-1}(T_1)$ and $\pi^{-1}(T_2)$.  Since $\pi$ is
ultimately constructed from a normalization \cref{Lstableexplicit}, every connected
component of $\pi^{-1}(T_i)$ maps surjectively onto $T_i$ for $i \in
\{1, 2\}$.  If $v_i$ is a leaf of $\Gamma_{\ol{\mc{Y}}}$ in $T_i$, then \cref{Lnonisom}(i)
shows that $|\pi^{-1}(v_i)| \leq n/2$, so there are at most $n/2$ connected components in
$\pi^{-1}(T_i)$.  So if $V$ (resp.\ $E$) is the number of vertices
(resp.\ edges) in $\Gamma$, then $V \leq E + n/2 + n/2 = E +
n$.  Since $V$ (resp.\ $E + n$) is the number of vertices (resp.\
edges) in $\Gamma_{\ol{\X}}$, this shows that the first homology of
$\Gamma_{\ol{\X}}$ has dimension at least $1$, so $\Gamma_{\ol{\X}}$
is not a tree.
\end{proof}

\begin{prop}\label{PBadReductionCurves}
None of the curves $\X_{17}$, $\X_{21}$, $\X_{42}$, $\X_{54}$, $\X_{59}$, or $\X_{61}$ in   \cref{Tcurvelist2} has CM Jacobian.
\end{prop}

\begin{proof}
For $\X_{17}$, the tree $\Gamma_{\ol{\mc{Y}}}$ as in \cref{Ptree} for the reduction modulo $3$ is shown in
\cite[Figure 3, p.\ 43]{MuellerThesis}.  In this case, $n = 5$, and
there are four edges which split the tree up into subtrees with $5$
and $15$ marked points.  By \cref{Ptree}, $\Gamma_{\ol{\X}_{17}}$ is
not a tree.  So $\Jac \X_{17}$ has bad
reduction, and thus does not have CM. 

For $\X_{61}$, the tree $\Gamma_{\ol{\mc{Y}}}$ for reduction modulo $5$ is shown in
\cite[Figure 11, p.\ 46]{MuellerThesis}.  In this case, $n = 7$, and
there are four edges which split the tree up into subtrees with $7$
and $35$ marked points.  Using \cref{Ptree} as before, $\Jac \X_{61}$ has bad
reduction, and thus does not have CM. 

For curves $\X_{21}$, $\X_{42}$, $\X_{54}$, and $\X_{59}$, the program
\texttt{stable\_reduction.sage}\footnote{available as
  part of the arxiv posting for this paper} gives the tree $\Gamma_{\ol{\mc{Y}}}$ for reduction modulo
$2$.
\begin{itemize}
  \item For curve $\X_{21}$, there is an edge splitting $\Gamma_{\ol{\mc{Y}}}$ into two trees with
$6$ and $12$ markings, and $n=3$.
  \item For curve $\X_{42}$, there is an edge splitting $\Gamma_{\ol{\mc{Y}}}$ into two trees with
$6$ and $24$ markings, and $n=3$.
  \item For curve $\X_{54}$, there is an edge splitting $\Gamma_{\ol{\mc{Y}}}$ into two trees with
$10$ and $40$ markings, and $n=5$.
  \item For curve $\X_{59}$, there is an edge splitting $\Gamma_{\ol{\mc{Y}}}$ into two trees with
    $6$ and $36$ markings, and $n=3$.
\end{itemize}
  In all cases $i \in \{21, 42, 54, 59\}$, using \cref{Ptree} as before shows that $\Gamma_{\ol{\X}_i}$ is not a
  tree, which means that $\Jac \X$ has bad
  reduction, and thus does not have CM.
\end{proof}

\begin{cor}\label{Ctree}
None of the curves $\X_{22}$, $\X_{23}$, $\X_{24}$, $\X_{44}$, $\X_{46}$,
$\X_{47}$, $\X_{56}$, or $\X_{63}$ in \cref{Tcurvelist2} has CM Jacobian.
\end{cor}

\begin{proof}
  The curves $\X_{22}$, $\X_{23}$, and $\X_{24}$ have $\X_{21}$ as a quotient (via an
  automorphism multiplying $y$ by an appropriate root of unity).
Likewise, the curves $\X_{44}$, $\X_{46}$, and $\X_{47}$
  have $\X_{42}$ as a quotient.  The curve $\X_{56}$ has
  $\X_{54}$ as a quotient.  The curve $\X_{63}$ has $\X_{61}$ as a
  quotient.

  Since all quotients of a CM curve must be CM curves,
  \cref{PBadReductionCurves} shows that none of the curves
  in the corollary is a CM curve.
\end{proof}

\subsection{Frobenius criterion}\label{Sfrobenius}
To show that the remaining curves in \cref{Tcurvelist2} do not
have CM, we use a criterion of M\"{u}ller--Pink.  Let $A$ be an
abelian variety defined over a number field $K$.  For any prime
$\mf{p}$ of $K$ where $A$ has good reduction, let $f_{\mf{p}} \in
\rats[T]$ be the minimal polynomial of the Frobenius at $\mf{p}$
acting on the Tate module of the reduction of $A$ modulo $\mf{p}$.
Let $E_{f_\mf{p}}$ equal $\rats[T]/f_{\mf{p}}$.  Since the Frobenius action is semisimple, the polynomial $f_{\mf{p}}$ has 
no multiple factors.

If $t$ is the image of $T$ in $E_{f_{\mf{p}}}$, then let 
$E'_{f_{\mf{p}}} \subseteq E_{f_{\mf{p}}}$ be the subring given by
intersecting the rings of $\rats[t^n]_{n \in \nats}$ inside
$E_{f_{\mf{p}}}$.  Observe that, if $f_{\mf{p}} = g(T^m)$ for some
polynomial $g$ and $m \in \nats$, we can replace $f_{\mf{p}}$ by
$g(T)$ when computing $E_{f_{\mf{p}}}'$.

The following criterion can be used to show that $A$ does not have CM.

\begin{prop}[{\cite[Theorem 6.2 (a) $\Rightarrow$ (b)]{muller-pink}}]\label{Pmullerpink}
Maintain notation as above.  If $\dim(A) = g$ and $A$ has CM, then
there exists a product of number fields $E$ with $\dim_{\rats}E \leq 2g$ such that for any good prime
$\mf{p}$, we have an embedding $E'_{f_{\mf{p}}} \hookrightarrow E$.
\end{prop}

\begin{rem}
The paper of M\"{u}ller--Pink uses the simpler (but weaker) criterion of
\cite[Corollary 6.7]{muller-pink}.  However, it appears to be
difficult to use this criterion to prove \cref{Pruleout} below.
\end{rem}

Before our main application of \cref{Pmullerpink}, we prove two lemmas.

\begin{lem}\label{Lproductsoffields}
If $K_1, \ldots, K_n$ and $L_1, \ldots, L_m$, are characteristic $0$ fields, then there
exists a $\rats$-algebra embedding of $K_1 \times \cdots \times K_n$
into $L_1 \times \cdots \times L_m$ if and only if there exists a
surjective map $\phi: \{1, \ldots, m\} \to \{1, \ldots, n\}$ such
that for all $i \in \{1, \ldots, m\}$, there exists an embedding
$\gamma_i: K_{\phi(i)} \hookrightarrow L_i$. 
\end{lem}

\begin{proof}
An embedding $\pi: K_1 \times \cdots \times K_n \hookrightarrow
  L_1 \times \cdots \times L_m$ gives rise to a $\rats$-algebra
  morphism $\pi_i: K_1 \times \cdots \times K_n \to L_i$ for each
  $i$.  Since the kernel is a prime ideal, this morphism is projection
  onto some $K_j$ followed by an embedding $K_j \hookrightarrow L_i$.
  Set $j = \phi(i)$.  The kernel of $\pi$ is the intersection of the
  kernels of the $\pi_i$, which is trivial only if each $j$ is equal to
  $\phi_i$ for some $i$, i.e., if $\phi$ is surjective.  Thus the
  condition in the lemma is necessary for the existence of an
  embedding.

Conversely, if the condition in the lemma is satisfied, then we define the  following $\rats$-algebra morphism, which is easily seen to be an embedding.
\begin{align*}
    \pi: K_1 \times \cdots \times K_n &\hookrightarrow L_1 \times                                        \cdots \times L_m \\
    (r_1, \ldots, r_n) &\mapsto (\gamma_i(r_{\phi(1)}), \ldots, \gamma_i(r_{\phi(m)})).
\end{align*}
\end{proof}

\begin{lem}\label{Lquotients}
The curves $\X_5$, $\X_6$, $\X_{26}$, $\X_{29}$, $\X_{33}$, $\X_{43}$,
and $\X_{66}$
from \cref{Tcurvelist2}
have quotients with the following affine equations, respectively,
where $i$ is a square root of $-1$: \\

\bigskip

\renewcommand\arraystretch{1.2}

\begin{tabular}{|l|l|}%[hbt]
\hline 
Curve & Affine Birational Equation of Quotient Curve \\
\hline
  $\X_5$ & $y^3 = x^6 - 33x^4 - 33x^2 + 1$ \\ \hline
  $\X_6$ &  $y^4 = x^6 - 33x^4 - 33x^2 + 1$ \\ \hline
  $\X_{26}$ & $y^7 = (x^4 - 4i x^2 + 12)(x^2 - 2i)$  \\ \hline
  $\X_{29}$ & $y^{13} = (x^2-4)^7(x+14)(x-34)$  \\ \hline
  $\X_{33}$ & $y^5 = x^{10} + 10x^8 + 35x^6 -228x^5 + 50x^4 - 1140x^3
  + 25x^2 -1140x + 496$  \\ \hline
  $\X_{43}$ & $y^5 = x^6 + 522x^5 - 10005x^4 - 10005x^2 - 522x + 1$\\
  \hline
 $\X_{66}$ & $y^{31} = (x^2 - 228x + 496)^2(x^2 + 522x -
             10004)^2(x+11)^2(x^2 + 4)$\\ \hline

\end{tabular}
\end{lem}

\medskip

\begin{proof}
  For $\X_5$ and $\X_6$, the affine equation given is the obvious
  quotient by the automorphism fixing $y$ and sending $x$ to $-x$.
  Likewise, for $\X_{43}$, the affine equation given is the obvious
  quotient by the automorphism fixing $y$ and sending $x$ to $\zeta_5
  x$.

  For $\X_{26}$, consider the order 2 automorphism $\sigma$ of
  $\complex(\X_{26})$ given by $\sigma(x) = i/x$ and $\sigma(y) =
  iy/x^2$.  The fixed subfield is generated by $z$ and $w$, where $z =
  x + i/x$ and $w = y/x$.  The affine equation of $\X_{26}$ is 
\[
y^7 =  x(x^4 - 1)(x^8 + 14x^4 + 1).
\]
In terms of $w$ and $z$, this becomes
  \begin{align*}
   w^7 &= \left(x^2 - \frac{1}{x^2}\right)\left(x^4 + 14 +
           \frac{1}{x^4}\right) \\
       &= (z^2 - 2i)(z^4 - 4iz^2 +  12),
  \end{align*}
as can be checked by hand or with a computer.  Changing back to $x$  and $y$ gives the equation in the table.

For $\X_{29}$, consider the automorphism group isomorphic to $D_8$  generated by $\sigma$ and $\tau$ where $\sigma(x,y) = (ix, iy)$ and $\tau(x, y) = (1/x, y/x^2)$.  
The fixed subfield is generated by $z$ and $w$, where $z = x^4 +  1/x^4$ and $w = y(x^8-1)/x^5$.  The affine equation of $\X_{29}$  is 
\[ y^{13} = x(x^8-1)(x^8+14x^4+1)(x^8-34x^4+1).\]
In terms of $w$ and $z$, this becomes
\[ w^{13} = (z^2-4)^7(z+14)(z-34),\]
as can be checked by hand.  Changing back to $x$  and $y$ gives the equation in the table.

For $\X_{33}$, consider the order 2 automorphism $\sigma$ of  $\complex(\X_{33})$ given by $\sigma(x) = -1/x$ and $\sigma(y) =  y/x^4$.  The fixed subfield is generated by $z$ and $w$, where $z =
  x - 1/x$ and $w = y/x^2$.  The affine equation of $\X_{33}$ is 
\[
y^5 =  x^{20} - 228x^{15} + 494x^{10} + 228x^5 + 1.
\]
  In terms of $w$ and $z$, this becomes
  \begin{align*}
   w^5 &= x^{10} - 228x^5 + 494 + \frac{228}{x^5} + \frac{1}{x^{10}} \\
       &= z^{10} + 10z^8 + 35z^6 -228z^5 + 50z^4 - 1140z^3
  + 25z^2 -1140z + 496,
  \end{align*}
as can be checked by hand or with a computer.  Changing back to $x$ and $y$ gives the equation in the table.

For $\X_{66}$, consider the automorphism group isomorphic to $D_{10}$  generated by $\sigma$ and $\tau$ where $\sigma(x,y) = (\zeta_5 x, \zeta_5 y)$ and $\tau(x, y) = (-1/x, y/x^2))$.  Here $\zeta_5$ is some primitive $5$th root of unity.      The fixed subfield is generated by $z$ and $w$, where $z = x^5 -  1/x^5$ and $w = y^2/x^2$.  The affine equation of $\X_{66}$   is 
\[
\begin{split}
y^{31}  = &x(x^{20} - 228x^{15} + 494x^{10} + 228x^5 + 1)(x^{30}  + 522x^{25} - 10005x^{20} - 10005x^{10} \\
 & - 522x^5 +1)(x^{10} + 11x^5  -1). \\
\end{split}
\]    
In terms of $w$ and $z$, this becomes
\[
w^{31} = (z^2 - 228z + 496)^2(z^2 + 4)(z^2 + 522z -  10004)^2(z+11)^2
\]             
as can be checked tediously by hand or more easily with some basic computational assistance.  Changing back to $x$  and $y$ gives the equation in the table.

\end{proof}

\begin{prop}\label{Pruleout}
  None of the curves $\X_2$, $\X_5$, $\X_6$, $\X_{26}$, $\X_{29}$,
  $\X_{33}$, $\X_{43}$, or $\X_{66}$
  in \cref{Tcurvelist2} has CM Jacobian.
\end{prop}

\begin{proof}
For $i \in \{5, 6, 26, 29, 33, 43, 66\}$, let $\X_i'$ be the quotient curve of
$X_i$ from \cref{Lquotients}.  It suffices to prove that neither
$\X_2$ nor any of these $\X_i'$ has CM.  For each curve, we use the
the program \texttt{frobenius\_polynomials.sage}\footnote{available as
  part of
 the arxiv posting for this paper} to compute the
algebras $E'_{f, \mf{p}}$ for various $p$, and then we show that it is
impossible for all the $E'_{f, \mf{p}}$ to embed into a
$\rats$-algebra of the correct dimension.  

For all cases other than
$\X_{29}'$ and $\X_{66}'$, the curve is $\X_i'$ is itself a
superelliptic curve , and we can compute the Frobenius minimal polynomial
using the superelliptic curve package in Sage (which only works on
these types of curves).  For $\X_{29}'$ and
$\X_{66}'$, we instead use the function \texttt{ZetaFunction} from
Magma, and take the reciprocal polynomial of radical of the numerator.  The results are in the
following charts.

\renewcommand\arraystretch{1.2}
\begin{longtable}{|l|l|}
  \caption{Curve $\X_2$, genus $g_2 := 16$} \\
  \hline $\mf{p}$ & $E'_{f, \mf{p}}$ \\
  \hline 7 & $\rats(\sqrt{-3}) \times \rats(\sqrt{-5 \cdot 23})$ \\
  \hline 13 & $\rats(\sqrt{-3})$ \\
  \hline 31 & $K_2 \times L_2$ \\
  \hline 43 & $\rats(\sqrt{-3}) \times \rats(\sqrt{-5 \cdot 127})$ \\
  \hline 67 & $\rats(\sqrt{-3}) \times \rats(\sqrt{-5 \cdot 11 \cdot 13})$ \\
  \hline  
\end{longtable}

Here, $[K_2:\rats] = [L_2:\rats] = 8$.
Additionally, one verifies that the only quadratic field
contained in $K_2$ is $\rats(\sqrt{5})$ and the only quadratic fields
contained in $L_2$ are those contained in $\rats(\sqrt{-3}, \sqrt{5})$.
Using \cref{Lproductsoffields}, one sees that a minimum-dimensional
product of number fields containing all these $E'_{f, \mf{p}}$ is
$$K_2(\sqrt{-3}) \times L_2 \times \rats(\sqrt{-3}, \sqrt{-5 \cdot 23}) \times
\rats(\sqrt{-3}, \sqrt{-5 \cdot 127}) \times \rats(\sqrt{-3}, \sqrt{-5
  \cdot 11 \cdot 13}),$$
which has dimension $16 + 8 + 4 + 4 + 4 > 32 = 2g_{2}$.  By \cref{Pmullerpink},
$\X_{2}$ does not have CM.

\begin{longtable}{|l|l|}
  \caption{Curve $\X_5'$, genus $g_5 := 4$} \\
  \hline $\mf{p}$ & $E'_{f, \mf{p}}$ \\
  \hline 7 & $\rats(\sqrt{-3})$ \\
  \hline 17 & $\rats \times \rats(\sqrt{-2})$ \\
  \hline 23 & $\rats \times \rats(\sqrt{-33})$ \\
  \hline 29 & $\rats \times \rats(\sqrt{-39})$ \\
  \hline  
\end{longtable}

Using \cref{Lproductsoffields}, one sees that a minimum-dimensional
product of number fields
containing all of these $E'_{f, \mf{p}}$ is
$$\rats(\sqrt{-2}, \sqrt{-3}) \times \rats(\sqrt{-3}, \sqrt{-33})
\times \rats(\sqrt{-3}, \sqrt{-39}),$$ which has dimension $12 > 8 = 2g_5$.
By \cref{Pmullerpink}, $\X_5'$ does not have CM.

\begin{longtable}{|l|l|}
  \caption{Curve $\X_6'$, genus $g_6 := 7$} \\
  \hline $\mf{p}$ & $E'_{f, \mf{p}}$ \\
  \hline 5 & $\rats(\sqrt{-1})$ \\
  \hline 7 & $\rats \times \rats(\sqrt{-6})$ \\
  \hline 17 & $\rats(\sqrt{-1}) \times \rats(\sqrt{-33})$ \\
  \hline 19 & $\rats(\sqrt{-1}) \times \rats(\sqrt{-21})$ \\
  \hline 29 & $\rats \times \rats(\sqrt{-13})$ \\
  \hline  
\end{longtable}

Using \cref{Lproductsoffields}, one shows that a minimum-dimensional
product of number fields containing all of these $E'_{f, \mf{p}}$ is
$$\rats(\sqrt{-1}, \sqrt{-6}) \times \rats(\sqrt{-1}, \sqrt{-13})
\times \rats(\sqrt{-1}, \sqrt{-21}) \times \rats(\sqrt{-1},
\sqrt{-33}),$$ which has dimension $16 > 14 = 2g_6$.
By \cref{Pmullerpink}, $\X_6'$ does not have CM.

\begin{longtable}{|l|l|}
  \caption{Curve $\X_{26}'$, genus $g_{26} := 15$} \\
  \hline $\mf{p}$ & $E'_{f, \mf{p}}$ \\
  \hline 5 & $\rats \times \rats(\sqrt{-2 \cdot 7 \cdot 11})$ \\
  \hline 17 &  $\rats \times \rats(\sqrt{-2 \cdot 5 \cdot 7 \cdot 13})$ \\
  \hline 29 & $\rats(\sqrt{-7}) \times \rats(\zeta_7) \times L_{26}$ \\
  \hline 37 & $\rats(\sqrt{-7}) \times \rats(\sqrt{-7})$\\
  \hline 61 & $\rats \times \rats(\sqrt{-2 \cdot 193})$ \\
  \hline 73 & $\rats \times \rats(\sqrt{-2 \cdot 7 \cdot 23 \cdot 113 \cdot 211 \cdot 1571})$ \\
  \hline  
\end{longtable}

Here $[L_{26} : \rats] = 12$ and the only quadratic field  contained in $L_{26}$ is $\rats(\sqrt{-7})$.    Using \cref{Lproductsoffields}, one shows that a minimum-dimensional  product of number fields containing all of these $E'_{f, \mf{p}}$ is
\[
\begin{split}
& \rats(\sqrt{-7}, \sqrt{-2 \cdot 7 \cdot 11}) \times \rats(\sqrt{-7}, \sqrt{-2 \cdot 5 \cdot 7 \cdot 13}) \times \rats(\sqrt{-7}, \sqrt{-2 \cdot 193}) \times  \\
& \times \rats(\sqrt{-7}, \sqrt{-a}) \times \rats(\zeta_7) \times L_{26}, \\
\end{split}
\]
 where $a = 2 \cdot 7 \cdot 23 \cdot 113 \cdot 211 \cdot 1571$.  This has dimension $34 > 30 = 2g_{26}$. By \cref{Pmullerpink}, $\X_{26}'$ does not have CM.

\begin{longtable}{|l|l|}
  \caption{Curve $\X_{29}'$, genus $g_{29} := 18$} \\
  \hline $\mf{p}$ & $E'_{f, \mf{p}}$ \\
  \hline 19 & $\rats \times \rats(\sqrt{-3 \cdot 7 \cdot 3847})$ \\
  \hline 53 & $K_{29} \times L_{29}$ \\
 \hline  
\end{longtable}

Here $[K_{29} : \rats] = 12$ and $[L_{29} : \rats] = 24$, and the only quadratic field
contained in $K_{29}$ or $L_{29}$ is $\rats(\sqrt{13})$.   
Using \cref{Lproductsoffields}, one shows that a minimum-dimensional
product of number fields containing $E'_{f, \mf{p}}$ for $p \in \{19, 53\}$ is
$$K_{29}(\sqrt{-3 \cdot 7 \cdot 3847}) \times L_{29},$$ which has dimension $48 > 36 = 2g_{29}$.
By \cref{Pmullerpink}, $\X_{29}'$ does not have CM.

\begin{longtable}{|l|l|}
  \caption{Curve $\X_{33}'$, genus $g_{33} := 16$} \\
  \hline $\mf{p}$ & $E'_{f, \mf{p}}$ \\
  \hline 7 & $\rats \times \rats(\sqrt{-6})$ \\
  \hline 13 &  $\rats \times \rats(\sqrt{-1})$ \\
  \hline 17 & $\rats \times \rats(\sqrt{-5 \times 41})$ \\
  \hline 37 & $\rats \times \rats(\sqrt{-3 \times 47})$\\
  \hline 43 & $\rats \times \rats(\sqrt{-3 \times 83})$ \\
  \hline 61 & $\rats(\zeta_5)$ \\
  \hline  
\end{longtable}

Using \cref{Lproductsoffields}, one shows that a minimum-dimensional
product of number fields containing all of these $E'_{f, \mf{p}}$ is
$$\rats(\zeta_5, \sqrt{-6}) \times \rats(\zeta_5, \sqrt{-1}) \times
\rats(\zeta_5, \sqrt{-5 \cdot 41}) \times
\rats(\zeta_5, \sqrt{-3 \cdot 47}) \times \rats(\zeta_5, \sqrt{-3 \cdot 83}).$$
This has dimension $40 > 32 = 2g_{33}$.
By \cref{Pmullerpink}, $\X_{33}'$ does not have CM.

\begin{longtable}{|l|l|}
  \caption{Curve $\X_{43}'$, genus $g_{43} := 10$} \\
  \hline $\mf{p}$ & $E'_{f, \mf{p}}$ \\
  \hline 7 & $\rats \times \rats(\sqrt{-10})$ \\
  \hline 11 & $K_{43} \times L_{43}$ \\
  \hline 13 & $\rats \times \rats(\sqrt{-3})$ \\
  \hline 17 & $\rats \times \rats(\sqrt{-2 \cdot 3 \cdot 7 \cdot 19})$ \\
  \hline  
\end{longtable}

Here, $[K_{43}:\rats] = 12$ and $[L_{43}:\rats] =
4$.   Additionally, one verifies that the only quadratic field
contained in $K_{43}$ is $\rats(\sqrt{-5})$ and the only quadratic field
contained in $L_{43}$ is $\rats(\sqrt{-43 \cdot 1361})$.  
Using \cref{Lproductsoffields}, a minimum-dimensional
product of number fields containing all these $E'_{f, \mf{p}}$ is
$$K_{43} \times L_{43} \times \rats(\sqrt{-3}) \times \rats(\sqrt{-10}) \times \rats(\sqrt{-2
  \cdot 3 \cdot 7 \cdot 19}),$$
which has dimension $22 > 20 = 2g_{43}$.  By \cref{Pmullerpink},
$\X_{43}'$ does not have CM.  This
completes the proof.

\begin{longtable}{|l|l|}
  \caption{Curve $\X_{66}'$, genus $g_{66} := 90$} \\
  \hline $\mf{p}$ & $E'_{f, \mf{p}}$ \\
  \hline 13 & $\rats \times L_{13}$ \\
  \hline 17 & $\rats \times L_{17}$ \\
  \hline 37 & $\rats \times L_{37}$ \\
  \hline 47 & $K_{47} \times L_{47}$ \\
  \hline  
\end{longtable}

Here, $[K_{13}:\rats] = [L_{13}:\rats] = 4$, and $[L_{37}:\rats] =
20$, $[K_{47}:\rats] = 6$, and $[L_{47}:\rats] = 30$.  Furthermore,
there is an embedding $K_{47} \hookrightarrow L_{47}$, the fields $L_{13}$ and
$L_{17}$ are linearly disjoint, and $K_{47}$ is linearly disjoint from
$L_i$ for $i \in \{13, 17, 37\}$.  All of this is verified in
\texttt{frobenius\_polynomials.sage}.  
Since $K_{47} \hookrightarrow L_{47}$, \cref{Lproductsoffields} shows that any product of
number fields containing all these $E'_{f, \mf{p}}$ must have an
embedding of $K_{47}$ into each factor.  Another application of
\cref{Lproductsoffields} shows that a minimum-dimensional product of
number fields containing all these $E_{f, \mf{p}}$ is
$$K_{47}L_{13} \times K_{47}L_{17} \times K_{47}L_{37} \times L_{47},$$
which has dimension $198 > 180 = 2g_{66}$.  By \cref{Pmullerpink},
$\X_{43}'$ does not have CM.  This
completes the proof.

\end{proof}

\begin{cor}\label{Cmullerpink}
None of the curves $\X_{3}$, $\X_{7}$, $\X_8$, or $\X_{45}$ in \cref{Tcurvelist2} has CM Jacobian.
\end{cor}

\begin{proof}
 The curve $\X_3$ has $\X_2$ as a quotient.  The curves $\X_7$ and $\X_8$ have $\X_5$ as a quotient, and the curve $\X_{45}$ has
  $\X_{43}$ as a quotient.  Since all quotients of a CM curve must be CM curves,
  \cref{Pmullerpink} shows that none of the curves
  in the corollary is a CM curve.
\end{proof}

Now we are ready to state the main theorem.  
Recall that if $\X$ is a smooth superelliptic curve 
with reduced automorphism group isomorphic to $A_{4}$, $S_{4}$, and
$A_{5}$, then $\X$ is isomorphic to one of the curves in 
\cref{Tcurvelist2}.  
\begin{thm}\label{Tmain} For each case in \cref{Tcurvelist2}, whether or not $\Jac \X$  has CM is determined in the 6-th column of the Table. 
\end{thm}

\begin{proof}
This follows from \cref{Pstreitexamples}, \cref{Phyperelliptic},
\cref{PBadReductionCurves}, \cref{Ctree},  \cref{Cmullerpink}, and \cref{Pruleout}.
\end{proof}

\begin{cor}\label{Cmain}
For superelliptic curves  with many automorphisms, having complex multiplication is equivalent to satisfying Streit's Criterion (\cref{Streit}).
\end{cor}

\begin{proof}
One direction is immediate from \cref{Streit}.  For the other, observe
first that for each entry in \cref{Tcurvelist2} that is a CM
curve, the fact that the Jacobian has CM is proven using Streit's
criterion in \cref{Pstreitexamples}.  Combining this with
\cref{PCmD2m} finishes the proof.
\end{proof}

\medskip

\noindent \textbf{Acknowledgments:} We would like to thank Andrew Sutherland and Richard Pink for their comments and suggestions when this paper was written. 
 
 %************************* 
 
\bibliographystyle{amsplain} 

\bibliography{CM}{}

\end{document}